\documentclass[11pt,a4paper]{amsart}
\usepackage[toc,page]{appendix}
\usepackage{graphicx} 
\usepackage{fix-cm} 
\usepackage{anyfontsize}
\usepackage{url}
\usepackage{amsmath} 
\usepackage{amssymb}
\usepackage{bm, bbm}
\usepackage{polynom}
\usepackage[outline]{contour}
\usepackage{xcolor}
\usepackage{pgfplots}
\usepackage{mathtools}
\usepackage{stmaryrd}
\usepackage[labelfont=rm]{subcaption}
\pgfplotsset{
  compat=1.13,
}
\usepackage{xfrac}  
\usepackage{enumitem} 
\usepackage{float}
\usepackage{comment}
\usepackage{mathrsfs}
\usepackage[ruled,vlined, algo2e, linesnumbered]{algorithm2e}

\usepackage{tikz}
\usepackage{tikz-cd}
\usepackage{scalerel,stackengine}

\usepackage[numbers]{natbib}
\theoremstyle{plain}
\newtheorem{thm}{Theorem}[section] 

\newtheorem{rem}[thm]{Remark}
\newtheorem{crl}[thm]{Corollary}

\newtheorem{prop}[thm]{Proposition}
\newtheorem{asm}[thm]{Assumption}

\contourlength{0.1pt}
\contournumber{10}%
\definecolor{clr1}{RGB}{27,158,119}
\definecolor{clr2}{RGB}{217,95,2}
\definecolor{clr3}{RGB}{117,112,179}
\definecolor{clr4}{RGB}{231,41,138}
\definecolor{clr5}{RGB}{102,166,30}
\definecolor{clr6}{RGB}{230,171,2}
\definecolor{clr7}{RGB}{166,118,29}
\definecolor{clr8}{RGB}{102,102,102}

\newcommand{\A}{\mathcal{H}}
\newcommand{\E}{\mathcal{E}}

\newcommand{\I}{\mathcal{I}}

\renewcommand{\P}{\mathcal{P}}
\newcommand{\R}{\mathcal{R}}
\newcommand{\T}{\mathcal{T}}

\newcommand{\HH}{H}
\newcommand{\HHH}{\bm{H}}
\newcommand{\LL}{L}
\newcommand{\LLL}{\bm{L}}

\newcommand{\pp}{p}

\newcommand{\St}{\mathrm{St}(p,\HHH)}

\newcommand{\Tan}{\mathrm{T}}
\newcommand{\Nrm}{\mathrm{N}}

\newcommand{\K}{{\mathbb K}}

\newcommand{\grad}{\mathrm{grad}}

\newcommand{\pf}{\mathrm{pf}}

\renewcommand{\sf}{\mathrm{Sf}}

\newcommand{\diag}{\mathrm{diag}}

\newcommand{\Herm}{\mathrm{Herm}}
\DeclareMathOperator{\trace}{tr}

\newcommand{\etabf}{{\bm \eta}}
\newcommand{\xibf}{{\bm \xi}}

\newcommand{\varphibf}{{\bm \varphi}}
\newcommand{\psibf}{{\bm \psi}}
\newcommand{\gbf}{\bm{g}}
\newcommand{\rbf}{\bm{r}}

\newcommand{\vbf}{\bm{v}}
\newcommand{\wbf}{\bm{w}}
\newcommand{\xbf}{\bm{x}}

\newcommand{\re}{\mbox{\rm Re}}

\newcommand{\Drm}{\mathrm{D}\,}
\newcommand{\drm}{\mathrm{d}}
\newcommand{\dx}{\mathrm{d}x}
\newcommand{\ea}[2]{{#1, \A, #2}}

\newcommand{\outerprod}[2]{\llbracket{#1},{#2}\rrbracket}

\makeatletter
\newsavebox{\@brx}
\newcommand{\llangle}[1][]{\savebox{\@brx}{\(\m@th{#1\langle}\)}%
  \mathopen{\copy\@brx\mkern2mu\kern-0.9\wd\@brx\usebox{\@brx}}}
\newcommand{\rrangle}[1][]{\savebox{\@brx}{\(\m@th{#1\rangle}\)}%
  \mathclose{\copy\@brx\mkern2mu\kern-0.9\wd\@brx\usebox{\@brx}}}
\makeatother
\newcommand{\coouterprod}[2]{\llangle{#1},{#2}\rrangle}

\usepackage[most]{tcolorbox}

\definecolor{myGreen2}{RGB}{114,175,30} 
\definecolor{myRed}{RGB}{180,50,50}  
\definecolor{myOrange}{RGB}{225,92,22} 

\definecolor{unia-purple}{RGB}{173, 0, 124}
\definecolor{light-gray}{rgb}{0.8889,0.8889,0.8889}
\definecolor{cpl1}{rgb}{0.8889,0.4356,0.2781}
\definecolor{cpl2}{rgb}{0.0,0.6056,0.9787}
\definecolor{cpl3}{rgb}{0.2422,0.6433,0.3044}
\definecolor{cpl4}{rgb}{0.2, 0.2, 0.2}

\newlength{\plotwidth}
\newlength{\plotheight}
\newcommand{\Field}{\mathbb C}

\usetikzlibrary{arrows.meta}
\usetikzlibrary{backgrounds}
\usepgfplotslibrary{patchplots}
\usepgfplotslibrary{fillbetween}
\pgfplotsset{%
    layers/standard/.define layer set={%
        background,axis background,axis grid,axis ticks,axis lines,axis tick labels,pre main,main,axis descriptions,axis foreground%
    }{
        grid style={/pgfplots/on layer=axis grid},%
        tick style={/pgfplots/on layer=axis ticks},%
        axis line style={/pgfplots/on layer=axis lines},%
        label style={/pgfplots/on layer=axis descriptions},%
        legend style={/pgfplots/on layer=axis descriptions},%
        title style={/pgfplots/on layer=axis descriptions},%
        colorbar style={/pgfplots/on layer=axis descriptions},%
        ticklabel style={/pgfplots/on layer=axis tick labels},%
        axis background@ style={/pgfplots/on layer=axis background},%
        3d box foreground style={/pgfplots/on layer=axis foreground},%
    },
}

\makeatletter
\def\author@andify{%
  \nxandlist {\unskip ,\penalty-1 \space\ignorespaces}%
    {\unskip {} \@@and~}%
    {\unskip \penalty-2 \space \@@and~}%
}
\renewcommand{\andify}{%
  \nxandlist{\unskip, }{\unskip{} \@@and~}{\unskip \space \@@and~}}
\makeatother

\title[Energy-Adaptive Riemannian CG for Kohn-Sham]{Energy-Adaptive Riemannian Conjugate Gradient Method for Density Functional Theory}

\author[D. Peterseim]{Daniel Peterseim}
\author[J. Püschel]{Jonas Püschel}
\author[T. Stykel]{Tatjana Stykel}
\address[D. Peterseim, T. Stykel]{Institute of Mathematics \& Centre for 
Advanced Analytics and Predictive Sciences (CAAPS), University of 
Augsburg, Universit\"atsstra{\ss}e~12a, 86159 Augsburg, Germany}
\email{daniel.peterseim@uni-a.de}
\email{tatjana.stykel@uni-a.de}
\address[J. Püschel]{Institute of Mathematics, University of Augsburg, 
Universit\"atsstra{\ss}e~12a, 86159 Augsburg, Germany}
\email{jonas.pueschel@uni-a.de}

\thanks{The work of D.~Peterseim is part of a project that has received funding from the European Research Council (ERC) under the European Union's Horizon 2020 research and innovation programme (Grant agreement No.~865751 -- RandomMultiScales).}

\begin{document}

\begin{abstract}
This paper presents a novel Riemannian conjugate gradient method for the Kohn-Sham energy minimization problem in density functional theory (DFT), with a focus on non-metallic crystal systems. We introduce an energy-adaptive metric that preconditions the Kohn-Sham model, significantly enhancing optimization efficiency. Additionally, a carefully designed shift strategy and several algorithmic improvements make the implementation comparable in performance to highly optimized self-consistent field iterations. The energy-adaptive Riemannian conjugate gradient method has a sound mathematical foundation, including stability and convergence, offering a reliable and efficient alternative for DFT-based electronic structure calculations in computational chemistry.
\end{abstract}

\maketitle



\section{Introduction}

The Kohn-Sham density functional theory (DFT) is a quantum chemistry method used to approximate the~ground state of a~molecular system, which is its lowest possible energy state~\cite{KohS65}. Despite the  limitations related to the inaccurate exchange-correlation approximation, its computational feasibility makes it successful, for example, in the field of condensed matter physics and solid-state materials \cite{Hasnip2014, Lev20}. For a molecular system with $\pp$ electrons, we consider the Kohn-Sham energy functional
\begin{equation}\label{eq:eks}
\begin{aligned}
\E(\varphibf) 
  = &\frac{1}{2} \int_\Omega  \sum\limits_{j=1}^\pp\|\nabla \varphi_j(x)\|^2\, \drm x
    + \int_\Omega \vartheta_{\rm ion}(x)\rho(\varphibf)(x)\, \drm x 
    \\&
    + \frac{1}{2} \int_\Omega\int_\Omega \frac{\rho(\varphibf)(x)\; \rho(\varphibf)(y)}{\|x-y\|} \, \drm y\,\drm x
    + \E_{\rm xc}(\rho(\varphibf))    
\end{aligned}
\end{equation}
for $\varphibf = (\varphi_1, \dots, \varphi_\pp)$ with orbitals $\varphi_j$
defined on a~spatial domain $\Omega \subseteq \mathbb R^3$ and the electron charge density $\rho(\varphibf) = |\varphi_1|^2+\ldots+|\varphi_\pp|^2$. For simplicity, we limit ourselves to orbitals \mbox{$\varphi_j \colon \Omega \to \Field$} without spin. The first term in \eqref{eq:eks} describes the total kinetic energy of the orbitals. The second term with the ionic potential $\vartheta_{\rm ion}$ is the ionic potential energy that characterizes the interaction of electrons with nuclei. The third term is the Hartree term, which can be interpreted as the Coulomb energy of the electrostatic average interaction. Finally, $\E_{\rm xc}$ models the exchange-correlation energy that balances the self-interaction included in the Hartree term; see \cite{GTH96, KohS65, PBE96, Tou23} for various choices.

To approximate a~ground state of a molecular system, we aim to compute the minimizer of the Kohn-Sham energy functional $\,\E$ over a suitable function space of orbitals subject to the orthonormality constraints
\begin{equation}\label{eq:orthonorm}
    \langle \varphi_i, \varphi_j \rangle_{L^2(\Omega, \Field)} = \delta_{ij},\qquad i,j = 1, \dots, \pp.
\end{equation}
These constraints ensure that the orbitals are pairwise independent and that each electron's probability density $|\varphi_j|^2$ integrates to $1$. Mathematically speaking, the energy is minimized over a Riemannian manifold known as the Stiefel manifold. This minimization problem can also be formulated on the Grassmann manifold \cite{AltPS24,SchRNB09}. The existence of minimal solutions has been studied in~\cite{Anantharaman2009}. 

The Kohn-Sham ground state is a solution of the nonlinear eigenvalue problem (NLEVP) 
\begin{equation}\label{eq:evphk}
    \A_{\varphibf}\, \varphibf = \varphibf\, \Lambda
\end{equation}
for the Kohn-Sham Hamiltonian 
\begin{equation}\label{eq:hks}
    \A_{\varphibf}\, \vbf = -\Delta \vbf + 2\vartheta_{\rm ion} \vbf + 2 \vartheta_H(\rho(\varphibf)) \vbf + 2 \vartheta_{\rm xc}(\rho(\varphibf))\vbf, 
\end{equation}
generated by the density $\rho(\varphibf)$. Here, the Laplacian $\Delta \vbf = (\Delta v_1 ,\dots, \Delta v_\pp)$ is applied com\-po\-nent-wise, $\vartheta_H(\rho) = \left(\rho \star \|\cdot\|^{-1}\right)$ is the Hartree potential, and $\vartheta_{\rm  xc}(\rho) = \frac{{\rm d} \E_{\rm xc}(\rho)}{{\rm d}\rho}$ is the exchange correlation potential. The entries of the Hermitian matrix $\Lambda$ are the Lagrange multipliers corresponding to the constraints~\eqref{eq:orthonorm}. 

In electronic structure calculations, the most popular numerical method for solving the Kohn-Sham problem is the self-consistent field (SCF) method~\cite{Roothaan1951}. This approach is based on the eigenvalue formulation and uses a fixed-point method to approximate the NLEVP~\eqref{eq:evphk} by a sequence of linear eigenvalue problems. Numerous performance-enhancing modifications, including preconditioning and acceleration strategies (e.g., \cite{Cancs2000, Kresse1996, Marks2008}), have been introduced, establishing SCF as the method of choice.

In recent years, however, direct minimization approaches for the energy functional~\eqref{eq:eks} subject to the orthonormality constraints~\eqref{eq:orthonorm} have attracted increasing attention. These methods are typically formulated as Riemannian optimization problems on the Stiefel and Grassmann manifolds. Originally introduced for finite-dimensional discretizations of the Kohn-Sham model in \cite{EAS99}, they have since been extended to the infinite-dimensional setting in \cite{AltPS22,AltPS24,SchRNB09}.

Within the class of Riemannian optimization methods, the Riemannian conjugate gradient (RCG) method is typically a prime choice. For the discretized Kohn-Sham problem, such methods have been presented in \cite{DaiLZZ17,EAS99,LuoWR24,Oviedo2018}, but as illustrated e.g.~in \cite{LuoWR24}, preconditioning is essential to maintain competitiveness with SCF as the discretization resolution increases. This observation is consistent with \cite{SchRNB09}, which demonstrates that preconditioning is strictly necessary for an~optimization method to be well defined in the infinite-dimensional case, that is, in the asymptotic limit as the spatial discretization becomes increasingly refined. 

Generic preconditioners, such as the one used in \cite{DdGYZ23} for metallic systems, offer a basic improvement in convergence, but do not fully exploit the inherent structure of the Kohn-Sham problem. In particular, they fail to adapt to the local energy landscape. In contrast, we propose a problem-adaptive preconditioning strategy that employs a variable Riemannian metric tuned to the energy. This energy-adaptive approach more accurately captures the local geometry of the energy functional, ensuring robust and efficient convergence even in the infinite-dimensional setting. Building on earlier work on ground state calculations for Bose-Einstein condensates using the Gross-Pitaevskii model~\cite{HenP20}, such an energy-adaptive preconditioner has been successfully implemented within a simple Riemannian gradient descent (RGD) method for the Kohn-Sham problem in \cite{AltPS22}, leading to the energy-adaptive Riemannian gradient descent (EARGD) method.

In this paper, we build upon EARGD and introduce an energy-adaptive Riemannian conjugate gradient (EARCG) scheme designed to further enhance practical performance. Moreover, we introduce several key improvements to the energy-adaptive RGD and RCG frameworks tailored for the Kohn–Sham energy minimization problem. Our contributions include a~shifting strategy to construct the energy-adaptive Riemannian metric from the Hamiltonian, a preconditioned block Krylov subspace method to calculate the energy-adaptive Riemannian gradient, a highly efficient step size control procedure, and an~adaptation of hybrid CG parameters. By integrating the Riemannian conjugate gradient method with an energy-adaptive metric for preconditioning and incorporating these algorithmic enhancements, the resulting EARCG method achieves high performance, making it competitive with state-of-the-art SCF-based schemes for periodic crystal semiconductors and insulators. Although our method is broadly applicable to related constrained energy minimization problems \cite{AltHPS24,BaoC13,BaoC18,TiaCWW20,VidNLC24}, in this work, we focus exclusively on the Kohn–Sham minimization problem.

The remaining parts of this paper are organized as follows. In Section~\ref{sec:ks}, we recall the Kohn-Sham model and discuss the geometry of the infinite-dimensional Stiefel manifold. The RCG method and all the necessary ingredients are introduced in Section~\ref{sec:rcg}. We construct the energy-adaptive metric and the respective EARCG scheme in Section~\ref{sec:earcg}. 
The aspects of discretization and (inexact) computation of the energy-adaptive Riemannian gradient as well as numerical experiments are discussed in Section~\ref{sec:experiments}. Using representative models in \texttt{DFTK.jl}~\cite{DFTKjcon}, we illustrate the competitiveness of EARCG with state-of-the-art SCF-based methods. Concluding remarks are collected in Section~\ref{sec:concl}.

\textbf{Notation.} For $p \in \mathbb N$, the set of $\pp\times \pp$ Hermitian matrices is denoted by $\Herm(\pp)$. The $\pp\times \pp$ identity and zero matrices are denoted by $I_p$ and $0_p$, respectively. Furthermore, $\trace M$ denotes the trace of $M\in\Field^{\pp\times \pp}$. For a complex number $z$, $\re(z)$ denotes the real part of $z$, and for $V\in\Field^{n\times p}$, $V^*$ denotes the conjugate transpose of $V$.

\section{Kohn-Sham energy minimization problem} \label{sec:ks}

In this section, we introduce the Kohn-Sham energy minimization problem constrained to the infinite-dimensional (complex) Stiefel manifold and briefly discuss its connection to the nonlinear eigenvalue problem. We also review the geometric structure of the Stiefel manifold, essential for constructing gradient-based Riemannian optimization methods. 

\subsection{Problem setting}
Let $\Omega\subset\mathbb{R}^3$ be a bounded convex Lipschitz domain. For $\K \in \left\{ \mathbb R, \mathbb C\right\}$, we consider a~closed subspace $\HH_\K  \subset H^1(\Omega, \Field)$ which is assumed to be dense in the Lebesgue space \mbox{$L_\K \coloneqq L^2(\Omega, \Field)$}. Here, the index $\K$ indicates whether we understand these spaces as real or complex Hilbert spaces. The complex inner products are defined as
\[
\langle v , w \rangle_{\LL_{\mathbb C}} 
    = \int_\Omega \overline{v}\,w \,\dx, \qquad\qquad 
\langle v, w \rangle_{\HH_{\mathbb C}}
    = \int_\Omega (\nabla v)^\ast \nabla w + \overline{v}\,w \,\dx,
\] 
and the real inner products are given by
\[
\langle v, w \rangle_{\LL_{\mathbb R}} = \re \big(\langle v, w \rangle_{\LL_{\mathbb C}}\big),\qquad\enskip \langle v, w \rangle_{\HH_{\mathbb R}} = \re \big( \langle v, w \rangle_{\HH_{\mathbb C}}\big). \qquad\quad
\]
Note that the real and complex inner products induce identical norms, and hence the corresponding Banach spaces coincide. We assume that $\HH_{\K} \subset \LL_{\K}\subset \HH_{\K}'$ forms a~Gelfand triple, where $\HH'_{\K}$ denotes the dual space of~$\HH_\K$ with respect to the $\LL_\K$-inner product. Once again, we have $\HH'_{\mathbb R} = \HH'_{\mathbb C}$ as Banach spaces, i.e., they contain the same elements and the norms are identical. For $p\in \mathbb{N}$, we consider the real Hilbert spaces $\HHH=\HH_{\mathbb R}^\pp$ and $\LLL =\LL_{\mathbb R}^\pp$ of $\pp$-frames. Since the Kohn-Sham energy functional $\,\E$ in \eqref{eq:eks} is real-valued, it is natural to consider these spaces as real Hilbert spaces.

Following~\cite{AltPS22}, for $\vbf=(v_1,\ldots,v_\pp), \wbf = (w_1,\ldots,w_p)\in \HHH$, we define the (complex) outer product on the pivot space $\LLL$ as
\[
\outerprod{\vbf}{\wbf} 
    = \big[\langle v_i, w_j\rangle_{\LL_{\mathbb C}}\big]_{i,j = 1}^{\pp,\pp} \in \Field^{\pp \times \pp}.
\]
It is sesquilinear, that is, conjugate-linear in the first argument and linear in the second argument, and hence $\mathbb R$-bilinear. Moreover, we have $\outerprod{\wbf}{\vbf} = \outerprod{\vbf}{\wbf}^\ast$. It is also compatible with matrix multiplication in the sense that $\outerprod{\vbf \,M}{\wbf} = M^\ast \outerprod{\vbf}{\wbf}$ and $\outerprod{\vbf}{\wbf\,M} = \outerprod{\vbf}{\wbf} M$ for all $M \in \Field^{\pp\times \pp}$. Furthermore, for $\vbf'\in\HHH'=(\HH'_{\mathbb R})^\pp$ and $\wbf\in\HHH$, we define
\[
\coouterprod{\vbf'}{\wbf} = \big[\langle v_i', w_j\rangle_{\HH_{\mathbb C}' \times \HH_{\mathbb C}^{}}\big]_{i,j = 1}^{\pp,\pp} \in \Field^{\pp \times \pp},
\]
to which we will refer to as the co-outer product. It represents the mutual (complex) duality pairing $\langle \cdot, \cdot\rangle_{\HH'_{\Field} \times \HH_{\Field}^{}}$  among all components. 
We note that the co-outer product satisfies the same matrix multiplication properties as those described for the outer product and that the outer and co-outer products coincide for $\vbf, \wbf \in \HHH$, i.e., $\outerprod{\vbf}{\wbf}  = \coouterprod{\vbf}{\wbf}$.

Define the (real) inner products on $\LLL$ and $\HHH$ by
\[
\langle \vbf, \wbf \rangle_{\LLL} 
    = \sum_{j=1}^p \langle v_j, w_j \rangle_{\LL_{\mathbb R}} 
    = \re \big( \trace \outerprod{\vbf}{\wbf} \big)
\quad\text{and}\quad
\langle \vbf, \wbf \rangle_{\HHH} 
= \sum_{j=1}^\pp \langle v_j, w_j \rangle_{\HH_{\mathbb R}},
\]
respectively, with induced norms $\|\vbf\|_{\LLL}=\sqrt{\langle \vbf, \vbf \rangle_{\LLL}}$ and $\|\vbf\|_{\HHH}=\sqrt{\langle \vbf, \vbf \rangle_{\HHH}}$.
For $\vbf' \in \HHH'$ and $\wbf \in \HHH$, the canonical (real) duality pairing is defined by
\[
\langle \vbf', \wbf \rangle
    = \sum_{j=1}^p \langle v_j', w_j \rangle_{\HH_{\mathbb R}' \times \HH_{\mathbb R}^{}}
    = \re \big( \trace \coouterprod{\vbf'}{\wbf} \big).
\]

For a $\pp$-electron system described by $\varphibf\in \HHH$, we consider the Kohn-Sham energy functional~$\,\E$ given in~\eqref{eq:eks}. The orthonormality constraints~\eqref{eq:orthonorm} induce the \emph{infinite-dimensional} (\emph{complex}) \emph{Stiefel manifold of index $\pp$}
\begin{equation}\label{eq:stiefel}
\St = \big\{\varphibf \in \HHH \enskip :\enskip \outerprod{\varphibf}{\varphibf}=I_\pp\big\}.
\end{equation}
A~{\emph{ground state}} $\varphibf_\star\in\HHH$ of the electron configuration is a global minimizer of the Kohn-Sham energy minimization problem 
\begin{equation}\label{eq:prob_e_min}
\min_{\varphibf\in \St} \E(\varphibf).
\end{equation}

Given the linear Hermitian Hamiltonian operator $\A_\varphibf:\HHH \to\HHH'$ from \eqref{eq:hks}, the real Fr\'echet derivative of $\E$ at $\varphibf \in \HHH$ in direction $\wbf\in\HHH$ is
\begin{align}\label{eq:derivative}
     \Drm \E (\varphibf)[\wbf] = \langle \A_{\varphibf}\,\varphibf,\wbf\rangle.
\end{align}
As a critical point of \eqref{eq:prob_e_min}, a~ground state $\varphibf_\star \in\St$ solves the NLEVP \eqref{eq:evphk}
with $\Lambda \in \Herm(\pp)$ being the Lagrange multiplier of the orthonormality constraints \mbox{$\outerprod{\varphibf}{\varphibf}\!=\!I_\pp$}. Although we view the feasibility set of \eqref{eq:prob_e_min} as a~real Hilbert manifold, the Lagrange multiplier may have complex entries. Note further that for any critical point $\varphibf\in\St$ satisfying \eqref{eq:evphk}, the associated Lagrange multiplier is given by $\Lambda = \coouterprod{\A_{\varphibf}\,\varphibf}{\varphibf}$. 

Since $\A_\varphibf$ acts component-wise on $\pp$-frames, it exhibits right associativity in the sense that for all $M\in\Field^{\pp\times \pp}$ and $\vbf\in\HHH$, $\A_\varphibf(\vbf M)=\A_\varphibf(\vbf)M$. Furthermore, both the Kohn-Sham energy functional~$\,\E$ and the corresponding Hamiltonian $\A_\varphibf$ are invariant under unitary transformations of the $\pp$-frame~$\varphibf$, i.e., for any unitary $U\in\Field^{\pp\times \pp}$, $\E(\varphibf U)=\E(\varphibf)$ and \mbox{$\A_{\varphibf U}=\A_\varphibf$}. This invariance implies that, upon diagonalizing the Lagrange multiplier $U^\ast\Lambda U=\diag(\lambda_1,\ldots,\lambda_\pp)$ with a~unitary $U$, the components of $\varphibf U$ are \emph{eigenfunctions} corresponding to the eigenvalues $\lambda_1,\ldots,\lambda_\pp$ of the component-wise action of $\A_{\varphibf}$. We say that a~ground state $\varphibf_\star$ of $\E$ satisfies the \emph{Aufbau principle}, if the eigenvalues of the associated Lagrange multiplier $\Lambda_\star=\coouterprod{\A_{\varphibf_\star}\varphibf_\star}{\varphibf_\star}$ are the $\pp$ smallest eigenvalues of the component-wise action of $\A_{\varphibf}$. In this case, the difference $\lambda_{p+1} - \lambda_p \ge 0$ is referred to as the \emph{highest occupied molecular orbital and lowest unoccupied molecular orbital} (HOMO-LUMO) gap.

\subsection{The infinite-dimensional Stiefel manifold}
The Stiefel manifold~$\St$ defined in \eqref{eq:stiefel} is a~smooth closed embedded submanifold of the Hilbert space $\HHH$ \cite[Prop.~3.1]{AltPS22} and thus inherits its topology from the ambient space $\HHH$. For any $\varphibf \in \St$, the \emph{tangent space} of $\St$ at $\varphibf$ is given by
\begin{equation}\label{eq:TanSt}
        \Tan_\varphibf\,\St =  \big\{ \etabf \in \HHH \enskip : \enskip \outerprod{\varphibf}{\etabf} + \outerprod{\etabf}{\varphibf}  = 0_p\big\}.
\end{equation}
We equip $\Tan_\varphibf\,\St$ with an~inner product $g_\varphibf(\cdot, \cdot)$. 
If this inner product varies smoothly on~$\varphibf$, 
then it induces a~\emph{Riemannian metric} that turns $\St$ into a~Riemannian Hilbert manifold.  

 For $\varphibf \in \St$, we define the \emph{Riemannian gradient}
 of the energy functional~$\,\E$ at $\varphibf$ with respect to the metric $g_\varphibf$ as the unique element $\grad\, \E(\varphibf) \in \Tan_\varphibf\, \St$ which satisfies the condition
 \begin{equation}\label{eq:defGrad}
     g_\varphibf(\grad\, \E(\varphibf), \etabf) = \Drm \E(\varphibf)[\etabf] \qquad \text{for all}\;\etabf\in \Tan_\varphibf\,\St.
 \end{equation}

For a~strong metric $g_\varphibf$ in the sense that $\Tan_\varphibf\, \St$ is complete with respect to the induced norm, the existence of the Riemannian gradient is guaranteed by the Riesz representation theorem. On the contrary, for a~weak metric, such as the $\LLL$-metric defined~by 
\[
 g_{\LLL}(\etabf,\xibf)=\langle\etabf,\xibf\rangle_{\LLL} \qquad \text{for all } \etabf,\xibf\in\Tan_\varphibf\,\St,
\]
the Riemannian gradient may not exist for all $\varphibf \in \St$.
Due to the Gelfand triple structure, we can, however, provide the corresponding element in the \emph{cotangent space} at $\varphibf \in \St$ given by
\[
\Tan_\varphibf^\ast\, \St = \big\{\vbf' \in \HHH' \enskip :\enskip\coouterprod{\vbf'}{\varphibf} + \coouterprod{\varphibf}{\vbf'} = 0_p\big\}.
\]
Here, we used the slight abuse of notation $\coouterprod{\varphibf}{\vbf'} \coloneqq \coouterprod{\vbf'}{\varphibf}^*$ to emphasize the similarity to the tangent space in~\eqref{eq:TanSt}.

Given $\varphibf\in\St$, we introduce the Hermitian matrix 
\begin{equation}\label{eq:Lambda}
    \Lambda_\varphibf = \coouterprod{\A_\varphibf\, \varphibf}{\varphibf},
\end{equation}
which can be interpreted as a matrix-valued Rayleigh quotient of $\varphibf$ with respect to $\A_\varphibf$. For critical points, it coincides with the Lagrange multiplier. Furthermore, we define the \emph{residual}
\[
    \mathrm{res}\,\E(\varphibf) = \A_\varphibf\, \varphibf - \varphibf\, \Lambda_\varphibf
\]
for the NLEVP~\eqref{eq:evphk} at $\varphibf$. It is easy to verify that the residual lies in the cotangent space $\Tan_\varphibf^\ast\, \St$.
Since for all $\etabf \in \Tan_\varphibf\, \St$, the matrix $\outerprod{\varphibf}{\etabf}$ is skew-Hermitian, we obtain with~\eqref{eq:derivative} that
\begin{equation}\label{eq:resDE}
\langle \mathrm{res}\,\E(\varphibf), \etabf \rangle 
    = \Drm\E(\varphibf)[\etabf]\qquad \text{for all}\;\etabf\in\Tan_\varphibf\,\St.
\end{equation}
This shows that $\mathrm{res}\,\E(\varphibf)$ is the element of $\Tan_\varphibf^\ast\, \St$  which represents the derivative of $\,\E$ at $\varphibf$. Under the regularity condition $\A_\varphibf\, \varphibf \in \HHH$, we have \mbox{$ \mathrm{res}\,\E(\varphibf) \in \Tan_\varphibf\, \St$}. Thus, $\mathrm{res}\,\E(\varphibf)$ can be understood as the Riemannian gradient with respect to the weak metric~$\langle \cdot, \cdot\rangle_{\LLL}$. Note that the combination of \eqref{eq:defGrad} and \eqref{eq:resDE} gives  
\begin{equation}\label{eq:calc_metric_g}
        g_{\varphibf}(\grad\,\E(\varphibf), \etabf) 
        = \Drm \E (\varphibf)[\etabf] 
        = \langle \mathrm{res}\,\E(\varphibf), \etabf \rangle_{} \quad\text{for all }\etabf \in \Tan_{\varphibf}\, \St,
    \end{equation}
which links the Riemannian gradient $\grad\,\E(\varphibf)$ with respect to a general metric $g_\varphibf$ to the residual $\mathrm{res}\,\E(\varphibf)$. 

\section{Riemannian conjugate gradient method} \label{sec:rcg}
In this section, we introduce the RCG method for minimizing the Kohn-Sham energy functional $\,\E$ on the infinite-dimensional Stiefel manifold $\St$ along with some underlying geometric concepts and algorithmic acceleration techniques. When designing and analyzing Riemannian optimization methods for partial differential equations, it is both natural and instructive to work in an infinite-dimensional framework. In this setting, the stability and convergence properties of optimization schemes are asymptotically independent of the spatial discretization resolution. In contrast, conventional convergence proofs in Riemannian optimization for discrete models often do not explicitly consider the dimensionality of the problem, which can limit their applicability as the number of degrees of freedom increases with finer spatial discretization. 

\begin{algorithm2e}[tbh]
\SetKwInOut{Input}{input}
\SetAlgoLined
\Input{initial guess $\varphibf^{(0)} \in \St$, metric $g_\varphibf$, retraction $\R_\varphibf$, vector transport  $\T^\R$
}
\BlankLine

$\rbf^{(0)} = \A_{\varphibf^{(0)}} \varphibf^{(0)} - \varphibf^{(0)}\coouterprod{\A_{\varphibf^{(0)}} \varphibf^{(0)}}{\varphibf^{(0)}}$ \tcc*{residual}
$\etabf^{(0)} = - \grad\, \E(\varphibf^{(0)})$ \tcc*{search direction}
 \For{$m=0, 1, 2, \dots$ until convergence}{
    compute the step size $\tau^{(m)}$\;
    $\varphibf^{(m+1)} = \R_{\varphibf^{(m)}}(\tau^{(m)} \etabf^{(m)})$\; 
    $\rbf^{(m+1)} = \A_{\varphibf^{(m+1)}} \varphibf^{(m+1)} - \varphibf^{(m+1)}\coouterprod{\A_{\varphibf^{(m+1)}}\varphibf^{(m+1)}}{\varphibf^{(m+1)}}$\;
    compute the CG parameter $\beta^{(m+1)}$\; 
    $\etabf^{(m+1)} = - \grad\, \E(\varphibf^{(m+1)}) + \beta^{(m+1)} \T_{\tau^{(m)} \etabf^{(m)}}^\R(\etabf^{(m)})$\; 
}
 \KwRet{$\varphibf^{(m+1)}$}
 \caption{RCG method for the Kohn-Sham energy minimization problem}
     \label{alg:rcg}
\end{algorithm2e}

The RCG method is a~straightforward modification of the classical nonlinear conjugate gradient method, e.g., \cite[Sect.~5.2]{Nocedal06}, and has been widely used for solving different optimization problems on (functional) manifolds \cite{AliHYY24,DaiLZZ17,EAS99,Oviedo2018,RiWi2012}. In addition to a~Riemannian metric, important ingredients of this method are
a~retraction, which maps tangent vectors to points on the manifold, and a~vector transport, which transfers vectors from one tangent space to another. In particular, for $\varphibf \in \St$ and $\etabf,\xibf\in \Tan_\varphibf\, \St$, a retraction $\R_\varphibf(\xibf)$ provides a~way to move from $\varphibf$ in a~direction $\xibf$ while staying on the Stiefel manifold $\St$, and a~vector transport $\T^\R_\xibf(\etabf)$ translates $\etabf$ to a~tangent vector at $\R_\varphibf(\xibf)$. 

In Algorithm~\ref{alg:rcg}, we present a~general formulation of the RCG method for the Kohn-Sham energy minimization problem~\eqref{eq:prob_e_min}. Notably, Algorithm~\ref{alg:rcg} computes residuals explicitly, as they allow the efficient evaluation of quantities of the form $g_\varphibf(\grad\,\E(\varphibf), \cdot)$ using \eqref{eq:calc_metric_g} which is often required in the adaptive step size control and for the computation of CG parameters. 

Different metrics, retractions, vector transports, step size strategies, and CG parameters in Algorithm~\ref{alg:rcg} will lead to different RCG schemes.  In the following, we discuss the particular choices for these algorithmic ingredients in more detail.

\subsection{Retraction}

A~retraction on the Stiefel manifold $\St$
can be constructed by using a~polar decomposition, defined for a~$\pp$-frame $\psibf\in\HHH$~by   
\begin{equation}\label{eq:polar}
\psibf = \pf(\psibf) \sf(\psibf)
\end{equation}
with the polar factor $\pf(\psibf)\in\St$ and the positive semidefinite matrix $\sf(\psibf)\in\Herm(\pp)$. If $\psibf$ has linearly independent components or, equivalently,  $\outerprod{\psibf}{\psibf}$ is positive definite, then the factors of the polar decomposition \eqref{eq:polar} can be computed from the eigenvalue decomposition $\outerprod{\psibf}{\psibf} = V D V^\ast$ with a~unitary matrix $V \in\Field^{\pp\times \pp}$ and a~diagonal matrix $D\in\mathbb{R}^{\pp\times \pp}$~as 
    \[
    \mathrm{pf}(\psibf) = \psibf \big(V D^{-1/2} V^\ast   \big) \qquad \text{and} \qquad \sf(\psibf) = V D^{1/2} V^\ast.
    \]

For $\varphibf\in \St$ and $\etabf\in\Tan_\varphibf\,\St$, we define the \emph{polar retraction}
\begin{equation}\label{eq:def_polar_retraction}
        \R_\varphibf^{\rm pol}(\etabf) = \mathrm{pf}(\varphibf + \etabf).
\end{equation}
Note that this retraction is well defined \cite{AltPS22} and of second order with respect to the $\LLL$-metric~\cite{AltPS24}. 

An~alternative retraction on $\St$ is based on the qR decomposition computed using the Gram-Schmidt orthonormalization procedure; see \cite{AltPS22} for details. Since our numerical experiments indicate a slightly better practical performance for the polar retraction than for the one based on the qR~decomposition, we will no longer consider the latter. 
    
\subsection{Vector transport} 
A~standard way to construct a~vector transport associated with a~retraction~$\R_\varphibf$ is by its differentiation, resulting in a~differentiated retraction vector transport 
\[
\T^{\R}_\xibf(\etabf)= \Drm\!\R_{\varphibf}(\xibf)[\etabf].
\]
To derive a~computable expression for such a~vector transport associated with the polar retraction $\R^{\rm pol}_\varphibf$ in \eqref{eq:def_polar_retraction}, we need to determine the derivative of the polar factor mapping $\pf$.

\begin{prop}\label{prop:proto_ortho_ret}
Let $\HHH_\star$ be the set of all $\pp$-frames from $\HHH$ with linearly independent components. For $\psibf \in \HHH_\star$, consider the polar decomposition \eqref{eq:polar}. Then \mbox{$\pf : \HHH_\star \to \St$} and $\sf: \HHH_\star \to \Herm(\pp)$ are both real Fr\'echet differentiable, and for all $\vbf \in \HHH$, the directional derivative of $\,\pf$ at $\psibf$ along $\vbf$ is given by
    \begin{equation}\label{eq:Df}
    \Drm\!\pf(\psibf) [\vbf] = \big(\vbf - \pf(\psibf)\, \Drm\! \sf(\psibf)[\vbf]\big) \sf(\psibf)^{-1},
    \end{equation}
    where the derivative $\Drm\! \sf(\psibf)[\vbf] \in \Herm(\pp)$ is the unique solution to the Lyapunov equation \begin{align}
        \sf(\psibf)\, X + X\, \sf(\psibf) = \outerprod{\psibf}{\vbf} + \outerprod{\vbf}{\psibf}. \label{eq:LyapDf}
    \end{align}
 \end{prop}
    
\begin{proof}
    Let $\psibf \in \HHH_\star$. 
    Calculating the directional derivatives of both sides of \eqref{eq:polar} at $\psibf$ along $\vbf\in\HHH$,  we obtain that
    \begin{align*}
        \vbf 
        & = \Drm \!\pf(\psibf)[\vbf]\, \sf(\psibf) + \pf(\psibf) \,\Drm\! \sf(\psibf)[\vbf]. 
    \end{align*}
    Resolving this equation for $\Drm\!\pf(\psibf)[\vbf]$ yields \eqref{eq:Df}.
    As $\Herm(\pp)$ is a~vector space, we have $\Drm\! \sf(\psibf)[\vbf] \in \Herm(\pp)$. Further, we calculate
    \begin{align*}
    \outerprod{\psibf}{\psibf} 
         = \outerprod{\pf(\psibf)\,\sf(\psibf)}{\pf(\psibf)\,\sf(\psibf)} 
        = \sf(\psibf) \outerprod{\pf(\psibf)}{\pf(\psibf)} \sf(\psibf)
        = \sf(\psibf) \sf(\psibf).  
    \end{align*}
    Differentiating both sides of this equation in direction $\vbf$ implies that 
    \[
       \outerprod{\psibf}{\vbf} + \outerprod{\vbf}{\psibf} =
       \sf(\psibf)\, \Drm\!\sf(\psibf)[\vbf] + \Drm\! \sf(\psibf)[\vbf]\,\sf(\psibf).
    \]
    Thus, $\Drm\!\sf(\psibf)[\vbf]$ satisfies \eqref{eq:LyapDf}. Since the matrix $\sf(\psibf)$ in the polar decomposition \eqref{eq:polar} is Hermitian and positive definite, the Lyapunov equation \eqref{eq:LyapDf} admits a~unique solution. 
\end{proof}

Using this proposition, we explicitly construct the differentiated polar retraction vector transport defined as $\T^{\rm pol}_\xibf(\etabf)=\Drm\!\R_\varphibf^{\rm pol}(\xibf)[\etabf]$. Setting $S=\sf(\varphibf + \etabf)$, we obtain from \eqref{eq:Df} with $\psibf=\varphibf+\xibf$ and $\vbf=\etabf$ that 
\[
    \T^{\rm pol}_\xibf(\etabf) = (\etabf - \R^{\rm pol}_\varphibf(\xibf)X)\,S^{-1},
\]
where $X$ solves the Lyapunov equation
$SX + XS = \outerprod{\xibf}{\etabf}+\outerprod{\etabf}{\xibf}$.
Note that for $\xibf=\tau \etabf$ with $\tau>0$, as required in step~8 of Algorithm~\ref{alg:rcg}, the solution of this equation is given by $X=\tau S^{-1}\outerprod{\etabf}{\etabf}$, and therefore,
\[
 \T^{\rm pol}_{\tau\etabf}(\etabf) = \etabf\, S^{-1} - \tau\,\R^{\rm pol}_\varphibf(\tau\etabf) \outerprod{\etabf\, S^{-1}}{\etabf\, S^{-1}}.
\]

Alternative vector transports developed for the matrix Stiefel manifold in \cite[Section~8.1]{Absil2007-ab} and \cite{Zhu2016,ZhuS20} can be extended to the infinite-dimensional setting in a~straightforward way.

\subsection{Step size strategy}\label{sec:step_size}

The convergence of the RCG method strongly depends on the chosen step sizes $\tau^{(m)}$. For $\varphibf^{(m+1)}=\R^{}_{\varphibf^{(m)}}(\tau^{(m)}\etabf^{(m)})$, we require them to satisfy the strong Wolfe conditions \cite{wolfe} composed of the Armijo condition
\[
    \E\big( \varphibf^{(m+1)}\big) \le \E(\varphibf^{(m)}) + \delta{} \,\tau^{(m)}  \,g_{\varphibf^{(m)}} \big( \grad\, \E(\varphibf^{(m)}), \etabf^{(m)} \big)
\]
and the strong curvature condition
\[
    \big| g_{\varphibf^{(m+1)}}\big(\grad\,\E(\varphibf^{(m+1)}), \Drm\!\R^{}_{\varphibf^{(m)}}(\tau^{(m)} \etabf^{(m)})[\etabf^{(m)}] \big)\big|  \le \sigma{} \big| g_{\varphibf^{(m)}} \big(\grad\,\E( \varphibf^{(m)}), \etabf^{(m)} \big)\big|,
\]
given constants $\delta{},\sigma{}>0$. Using the residual $\rbf^{(m)}=\mathrm{res}\,\E(\varphibf^{(m)})$ and the differentiated retraction transport $\T_{\tau^{(m)}\etabf^{(m)}}^{\R}(\etabf^{(m)})=\Drm\!\R_{\varphibf^{(m)}}(\tau^{(m)}\etabf^{(m)})[\etabf^{(m)}]$, we derive from  \eqref{eq:calc_metric_g} the equivalent formulation of these conditions 
\begin{align} \label{eq:armijo_cond}
\E\big( \varphibf^{(m+1)}\big)  & \le \E( \varphibf^{(m)}) + \delta{} \,\tau^{(m)}  \,\langle \rbf^{(m)}, \etabf^{(m)} \rangle, \\
 \label{eq:strong_curvature}
\big| \langle \rbf^{(m+1)}, \T^{\R}_{\tau^{(m)}\etabf^{(m)}}(\etabf^{(m)}) \rangle\big|  & \le \sigma{} \big| \langle \rbf^{(m)}, \etabf^{(m)} \rangle\big|.
\end{align}
Note that for a given search direction, the strong Wolfe conditions do not depend on the metric. This is natural since the metric does not influence the updated point $\varphibf^{(m+1)}$ or the values of $\E$ and its derivative $\Drm\E$ evaluated at $\varphibf^{(m+1)}$.

In order to compute step sizes satisfying the conditions~\eqref{eq:armijo_cond} and \eqref{eq:strong_curvature}, we use an~adaptive backtracking strategy augmented with a secant step as described in \cite{HZBacktrack}. This secant step corresponds to one iteration of the secant method~\cite{PaTa13}. 
In practice, checking the Armijo condition \eqref{eq:armijo_cond} as described in~\cite{HZBacktrack} has proven persistently problematic. To overcome these difficulties, we check the strong formulation of the approximate Wolfe conditions~\cite[Condition~T1]{HZBacktrack} instead, which read
\begin{equation}\label{eq:approx_wolfe}
    \begin{aligned}
    \E\big( \varphibf^{(m+1)}\big) &\le (1 + \epsilon)\E(\varphibf^{(m)}),\\
        \big| \langle \rbf^{(m+1)}, \T^{\R}_{\tau^{(m)}\etabf^{(m)}}(\etabf^{(m)}) \rangle\big|  &\le \mathrm{min} \{\sigma{}, 1 - 2 \delta \} \big| \langle \rbf^{(m)}, \etabf^{(m)} \rangle\big|          
    \end{aligned}
\end{equation}
with the parameters satisfying $0<\delta < \tfrac{1}{2}$ and $\delta < \sigma < 1$, and a small $0<\epsilon \ll 1$ to account for inaccuracies in the energy calculation.
In Algorithm~\ref{alg:mod_secant_step}, we present the modified backtracking step size strategy. In every RCG iteration, the backtracking process is initiated with the step size $\tau^{(k-1)}$ from the previous iteration, and we set $\tau^{(-1)}=1$. 

\begin{algorithm2e}[t]
\SetKwInOut{Input}{input}
\SetKwRepeat{Do}{do}{while}
\SetAlgoLined
\SetNoFillComment 
\Input{$f(t) \coloneqq \E \big( \R^{}_{\varphibf^{(m)}}(t\,\etabf^{(m)})\big)$ with $\varphibf^{(m)} \in \St$ and $\etabf^{(m)} \in \Tan_{\varphibf^{(m)}}\St$, backtracking parameters $0 < \delta{} < \tfrac{1}{2}$, $\delta < \sigma < 1$, $0 < \gamma < 1$, $0<\epsilon \ll 1$, \\ initial guess~$\tau > 0$} 
\BlankLine
$[a,b] = [0, \tau]$\;
\While{$\tau$ does not suffice \eqref{eq:approx_wolfe}}{
    $\displaystyle{\tau = \frac{a f'(b) - b f'(a)}{f'(b) - f'(a)}}$ \tcc*{secant step}
    \uIf{$f(\tau) > (1 + \epsilon) f(0)$ and $f'(\tau) < 0$}{
        $\tau = \gamma b$ \hfill\tcc{if step size is too big, decrease search interval}
        $[a,b] = [0, \tau]$\;
    }
    \uElseIf{$\tau < a$}{
        $\tau = \tfrac{1}{\gamma} b$  \hfill\tcc{if step size is too small, increase search interval}
        $[a,b] = [0, \tau]$\;
    }
\uElseIf{$\tau > b$}{
        $[a,b] = [b, \tau]$ \hfill\tcc{if $\tau$ is right of the interval, shift to the right} 
    }
    \ElseIf{$a \le \tau \le b$}{
    \uIf{$f'(\tau) < 0$}{
           $[a,b] = [a, \tau]$ \hfill\tcc{if $f'$ changes sign in $[a, \tau]$, take it}
    }
    \Else{
         $[a,b] = [\tau, b]$ \hfill\tcc{otherwise, take $[\tau, b]$}
    }
  }
}
\KwRet{$\tau$}
 \caption{Modified secant step size strategy}
     \label{alg:mod_secant_step}
\end{algorithm2e}

\subsection{Calculation of CG parameters}
The CG parameters $\beta^{(m+1)}$ significantly influence the performance of the RCG method. In unconstrained optimization, there is a wide array of such parameters (see, e.g.~\cite{HagZ06}), and although they are equivalent in the case of a quadratic energy functional, their behavior may differ for more general problems. 

For Riemannian optimization on matrix manifolds, several different CG parameters have been presented in \cite{SAKAI2023127685}. Their extension to the infinite-dimensional setting is straightforward. In fact, using \eqref{eq:calc_metric_g}, most of the CG parameters can be rewritten so that they become independent of the specific metric and only involve the calculation of the canonical pairing~$\langle\cdot,\cdot\rangle_{{}}$ for the dual evaluation. Following \cite{SaIi21}, we employ the hybrid Fletcher-Reeves Polak-Ribi\`ere-Polyak (FR-PRP) parameters \cite{Hu1991} given by
\[
    \beta^{(m+1)}_{\rm FR-PRP} = \max\big\{0, \min\{\beta^{(m+1)}_{\rm FR}, \beta^{(m+1)}_{\rm PRP}\}\big\}
\]
with the Fletcher-Reeves parameters \cite{FR}
\[
    \beta^{(m+1)}_{\rm FR} = \frac{\langle \rbf^{(m+1)}, \grad\, \E(\varphibf^{(m+1)}) \rangle_{}}{\langle \rbf^{(m)}, \grad\,\E(\varphibf^{(m)}) \rangle_{}} 
\]
and the Polak-Ribi\`ere-Polyak parameters \cite{PolakRibiere, POLYAK196994}
\[ 
\beta^{(m+1)}_{\rm PRP} = \frac{\langle \rbf^{(m+1)}, \grad \,\E(\varphibf^{(m+1)}) - \T^{\R}_{\tau^{(m)} \eta^{(m)}} \big(\grad \,\E(\varphibf^{(m)})\big) \rangle_{}}{\langle \rbf^{(m)}, \grad\,\E(\varphibf^{(m)}) \rangle_{}}.
\] 
Note that we slightly modified the usual formulas in order to be able to apply the simplifying relation \eqref{eq:calc_metric_g}.

\begin{rem}[Convergence analysis]
The convergence of the RCG method has been analyzed in~\textup{\cite{RiWi2012,SaIi21,SAKAI2023127685}} in both finite- and infinite-dimensional settings. More specifically, under assumptions that the cost functional $\,\E$ is bounded from below,
the Riemannian gradient $\grad\,\E(\varphibf)$ is Lipschitz continuous with respect to the retraction $\R_\varphibf$, and that 
step sizes $\tau^{(m)}$ satisfy the strong Wolfe conditions~\eqref{eq:armijo_cond} and~\eqref{eq:strong_curvature} with $0 < \delta{} < \sigma{} < \tfrac{1}{2}$, it has been shown in \textup{\cite{SaIi21}} that the RCG method with the hybrid FR--PRP parameters $\beta_{\rm FR-PRP}^{(m+1)}$, and the scaled differentiated retraction vector transport (where scaling prevents the transported vector from growing in the metric norm) 
generates a~sequence $(\varphibf^{(m)})_{m \in\mathbb N}$ that satisfies
\[
    \liminf\limits_{m \to \infty} \| \grad\, \E(\varphibf^{(m)}) \|_{\varphibf^{(m)}}^2 = 0.
\]
This result, together with the continuity of $\grad\, \E(\varphibf)$, implies that if the sequence $(\varphibf^{(m)})_{m \in\mathbb N}$ converges, its limit is a~critical point.
The convergence proof in \textup{\cite{SaIi21}}, which is a~modified version of that in~\textup{\cite{RiWi2012}}, is by contradiction to Zoutendijk's theorem \textup{\cite{zoutendijk}}. All arguments, as~well as the proof of Zoutendijk's theorem, are applicable to the infinite-dimensional case, as already discussed in~\textup{\cite{RiWi2012}}.
\end{rem}

\section{Energy-adaptive optimization}\label{sec:earcg}
In this section, we introduce an energy-adaptive metric derived from a shifted Hamiltonian and develop the corresponding energy-adaptive RCG method. We also discuss an adaptive strategy for computing the shifts that ensures the coercivity of the shifted Hamiltonian at each RCG iteration.

\subsection{Energy-adaptive metric}\label{sec:ea-metric}
As observed in \cite{MisS16,RiWi2012}, the choice of metric is crucial for the computational efficiency of Riemannian optimization methods. Defining an~adaptively changing metric, which exploits the first-order information of the energy functional, is particularly advantageous, as it exhibits a~preconditioning effect and significantly accelerates the convergence of the algorithms. Such an~energy-adaptive metric has first  been introduced in~\cite{HenP20} for the Gross-Pitaevskii minimization problem and then extended in~\cite{AltPS22} to Kohn-Sham-type models with a~coercive Hamiltonian.  However, the Kohn-Sham Hamiltonian $\A_\varphibf$ in~\eqref{eq:hks} is not coercive and, hence, cannot be employed directly to define a~metric on $\St$. This difficulty can be overcome by using an appropriate shift.  
For $\varphibf\in\St$  and a~shift matrix $\Sigma_\varphibf \in \Herm(\pp)$ depending smoothly on $\varphibf$, we define the \emph{shifted Hamiltonian}
\begin{equation}\label{eq:shiftedHam} 
    \A_{\varphibf,\Sigma_\varphibf}(\vbf) = \A_\varphibf\, \vbf - \vbf\, \Sigma_\varphibf \qquad \text{for all }\vbf\in\HHH,
\end{equation}
and in case it is bounded and coercive, the (strong) \emph{energy-adaptive Riemannian metric}
\begin{equation}\label{eq:ea_metric}
 g_{\ea{\varphibf}{\Sigma_\varphibf}}(\eta, \xibf) 
 = \langle \A_{\varphibf, \Sigma_\varphibf}(\etabf), \xibf \rangle
 \qquad \text{for all }\etabf,\xibf\in \Tan_\varphibf\, \St.
 \end{equation}
Note that, unlike the Hamiltonian $\A_\varphibf$, the shifted Hamiltonian $\A_{\varphibf,\Sigma_\varphibf}$ is no longer right associative. Therefore, we must explicitly enclose its argument in parentheses.

\begin{rem}\label{rem:Garding} 
In \cite[Lemma~5.1]{AltPS22}, it has been shown that under some boundedness conditions on $\vartheta_{\rm ion}$ and~$\E_{\rm xc}$, the linear Hermitian operator $\A_\varphibf$ is uniformly bounded
and there exists a~$\sigma_\varphibf>0$ such that $\A_{\varphibf,\Sigma_\varphibf}$ with the uniform shift $\Sigma_\varphibf = - \sigma_\varphibf I_p$ is coercive with the coercivity constant $1$.
\end{rem}

\subsection{Energy-adaptive gradient}\label{sec:ea-gradient}
We now derive an~expression for the Riemannian gradient of the Kohn-Sham energy functional $\,\E$ with respect to the energy-adaptive metric $g_{\ea{\varphibf}{\Sigma_\varphibf}}$ defined in \eqref{eq:ea_metric}, which for short is called the \emph{energy-adaptive Riemannian gradient}. 
This requires the use of an appropriate choice of the shift $\Sigma_\varphibf$, as specified in the following assumption.
\begin{asm}\label{asm:shift}
    A~shift $\Sigma_\varphibf\in\Herm(p)$ is such that the shifted Hamiltonian $\A_{\varphibf,\Sigma_\varphibf}$ is coercive with respect to $\|\cdot\|_{\HHH}$ on $\Tan_\varphibf\, \St$ and invertible on $\HHH$.
\end{asm}
This assumption ensures that the energy-adaptive metric is positive definite and the following projection is well defined.
The existence of such a~shift will be discussed in \mbox{Section~\ref{sec:shift}}. Under Assumption~\ref{asm:shift}, for $\varphibf \in \St$, we define an~\emph{energy-adaptive projection}
$\P_\varphibf : \HHH \to \Tan_\varphibf\, \St$ via the variational equality
\begin{equation}\label{eq:defP}
g_{\ea{\varphibf}{\Sigma_\varphibf}}(\P_\varphibf(\vbf), \etabf ) = \langle \A_{\varphibf, \Sigma_\varphibf}(\vbf), \etabf \rangle \qquad \text{for all }\vbf\in \HHH, \;\etabf \in \Tan_\varphibf\, \St.
\end{equation}
Note that the existence and uniqueness of $\P_\varphibf(\vbf)$ is guaranteed by the Riesz representation theorem. In the case where $\A_{\varphibf, \Sigma_\varphibf}$ is coercive on $\HHH$, $\P_\varphibf$ is exactly the orthogonal projection with respect to the induced inner product $g_{\ea{\varphibf}{\Sigma_\varphibf}}$ on $\HHH$. The following proposition provides an~expression for the projection $\P_\varphibf$.

\begin{prop}\label{prop:projP}
Let Assumption~\textup{\ref{asm:shift}} be fulfilled. For $\varphibf \in \St$ and $\vbf \in \HHH$, the energy-adaptive projection is given by 
\begin{equation}\label{eq:projP}
\P_\varphibf(\vbf) = \vbf - \A_{\varphibf, \Sigma_\varphibf}^{-1}(\varphibf X_{\vbf}),
\end{equation}
where $X_{\vbf} \in \Herm(p)$ is the unique solution to 
\begin{equation}\label{eq:eqXv}
\outerprod{ \A_{\varphibf, \Sigma_\varphibf}^{-1}(\varphibf X_{\vbf})}{\varphibf} + \outerprod{\varphibf}{\A_{\varphibf, \Sigma_\varphibf}^{-1}(\varphibf X_{\vbf})} = \outerprod{\vbf}{\varphibf} + \outerprod{\varphibf}{\vbf}.
\end{equation}
\end{prop}
\begin{proof}
Let $\varphibf \in \St$ and $\vbf \in \HHH$.
By definition, for all $\etabf \in \Tan_\varphibf\, \St$, we have 
\[
\langle \A_{\varphibf, \Sigma_\varphibf}(\P_\varphibf(\vbf)), \etabf \rangle = \langle \A_{\varphibf, \Sigma_\varphibf}(\vbf), \etabf \rangle. 
\]
Due to the linearity of $\A_{\varphibf, \Sigma_\varphibf}$, this implies that $\A_{\varphibf, \Sigma_\varphibf}(\vbf - \P_\varphibf (\vbf)) \in \Nrm_\varphibf^*\,\St$, where 
    \[
    \Nrm_\varphibf^*\,\St = \big\{\vbf' \in \HHH' \enskip : \enskip \langle \vbf' , \etabf \rangle = 0 \;\text{ for all } \etabf \in \Tan_\varphibf\, \St\big\}
    \]
is the conormal space of $\St$ at $\varphibf$ with respect to $\langle\cdot, \cdot \rangle$. The definition of the cotangent space $\Tan^\ast_\varphibf\, \St$ yields then the orthogonal decomposition     $\HHH' = \Tan^\ast_\varphibf\, \St \oplus \Nrm_\varphibf^*\,\St$ 
which, as the dimension of $\Tan^\ast_\varphibf\, \St$ is finite, gives 
    \[
    \Nrm_\varphibf^*\,\St = \big\{\varphibf X \enskip : \enskip X \in\Herm(p)\big\}.
    \]
Thus, there exists $X_{\vbf} \in \Herm(p)$ such that $\A_{\varphibf, \Sigma_\varphibf}(\vbf - \P_\varphibf (\vbf)) = \varphibf X_{\vbf}$.     Resolving this equation for $\P_\varphibf(\vbf)$, we obtain \eqref{eq:projP}. Since $\P_\varphibf(\vbf) \in \Tan_\varphibf\, \St$, we know that
    \[
    0_p = \outerprod{\P_\varphibf(\vbf)}{\varphibf} + \outerprod{\varphibf}{\P_\varphibf(\vbf)} = \outerprod{\vbf}{\varphibf} - \outerprod{\A_{\varphibf, \Sigma_\varphibf}^{-1}(\varphibf X_{\vbf})}{\varphibf} + \outerprod{\varphibf}{\vbf} - \outerprod{\varphibf}{\A_{\varphibf, \Sigma_\varphibf}^{-1}(\varphibf X_{\vbf})},
    \]
which can be re-arranged into \eqref{eq:eqXv}. It remains to show that this equation has a unique Hermitian solution. For that, let $Y_{\vbf}\in\Herm(p)$ solve equation \eqref{eq:eqXv} and let $\xibf = \vbf -  \A_{\varphibf, \Sigma_\varphibf}^{-1}(\varphibf Y_{\vbf}) $. Then we have
\begin{equation*}
    \begin{aligned}
        \outerprod{\xibf}{\varphibf} + \outerprod{\varphibf}{\xibf} = \outerprod{\vbf}{\varphibf} - \outerprod{\A_{\varphibf, \Sigma_\varphibf}^{-1}(\varphibf Y_{\vbf})}{\varphibf} + \outerprod{\varphibf}{\vbf} - \outerprod{\varphibf}{\A_{\varphibf, \Sigma_\varphibf}^{-1}(\varphibf Y_{\vbf})}
        = 0_p,
    \end{aligned}
\end{equation*}
and, hence, $\xibf \in \Tan_\varphibf\, \St$.  Additionally, for all $\etabf \in \Tan_\varphibf\, \St$, we obtain that 
\begin{equation*}
    \begin{aligned}
        g_{\ea{\varphibf}{\Sigma_\varphibf}}(\xibf, \etabf )
        & =  \langle \A_{\varphibf, \Sigma_\varphibf}(\vbf -  \A_{\varphibf, \Sigma_\varphibf}^{-1}(\varphibf Y_{\vbf})), \etabf \rangle \\
        & = \langle \A_{\varphibf, \Sigma_\varphibf}(\vbf), \etabf\rangle - \re\big(\trace \big(Y_{\vbf}\outerprod{\varphibf }{\etabf}\big)\big) 
        = \langle \A_{\varphibf, \Sigma_\varphibf}(\vbf), \etabf\rangle.
    \end{aligned}
\end{equation*}
in the last step, we used the fact that the trace of the product of the Hermitian and skew-Hermitian matrices is imaginary. 
    From the definition and uniqueness of $\P_\varphibf(\vbf)$ follows that $\xibf = \P_\varphibf(\vbf)$ and, consequently, 
    $\vbf -  \A_{\varphibf, \Sigma_\varphibf}^{-1}(\varphibf Y_{\vbf}) = \vbf -  \A_{\varphibf, \Sigma_\varphibf}^{-1}(\varphibf X_{\vbf})$.
    The bijectivity of $\A_{\varphibf, \Sigma_\varphibf}^{-1}$ gives 
    $\varphibf Y_{\vbf} = \varphibf X_{\vbf}$. Finally, by computing  the outer product with $\varphibf$, we find that
    $Y_{\vbf} = \outerprod{\varphibf}{\varphibf Y_{\vbf}} = \outerprod{\varphibf}{\varphibf X_{\vbf}} = X_{\vbf}$.
\end{proof}

Using Proposition~\ref{prop:projP}, we derive the following representation for the energy-adaptive Riemannian gradient of $\,\E$, denoted by $\grad_{\ea{\varphibf}{\Sigma_\varphibf}} \,\E(\varphibf)$.

\begin{prop}\label{prop:ea_grad_shift_general}
Let Assumption~\textup{\ref{asm:shift}} be fulfilled.
    The energy-adaptive Riemannian gradient of the functional $\E$ in \eqref{eq:eks} at $\varphibf \in \St$ is given by
    \begin{equation}\label{eq:eagrad_raw}
    \grad_{\ea{\varphibf}{\Sigma_\varphibf}} \,\E(\varphibf) =  \varphibf - \A_{\varphibf, \Sigma_\varphibf}^{-1} (\varphibf X_{\varphibf}),
    \end{equation}
    where $X_{\varphibf} \in \Herm(p)$ is the unique solution to the matrix equation
    \begin{equation}\label{eq:Xphi}
        \outerprod{ \A_{\varphibf, \Sigma_\varphibf}^{-1}(\varphibf X_{\varphibf})}{\varphibf} + \outerprod{\varphibf}{\A_{\varphibf, \Sigma_\varphibf}^{-1}(\varphibf X_{\varphibf})} = 2 I_p.
    \end{equation}
\end{prop}
\begin{proof} First, we show that 
    $\grad_{\ea{\varphibf}{\Sigma_\varphibf}} \,\E(\varphibf) = \P_\varphibf(\varphibf)$.
    Obviously, $\P_\varphibf(\varphibf) \in \Tan_\varphibf\, \St$. Additionally, for all $\etabf \in \Tan_\varphibf\, \St$, we obtain with~\eqref{eq:defP}, \eqref{eq:shiftedHam} and \eqref{eq:derivative} that  
    \[
    g_{\ea{\varphibf}{\Sigma_\varphibf}}(\P_\varphibf(\varphibf), \etabf ) 
    = \langle \A_{\varphibf,\Sigma_\varphibf}(\varphibf), \etabf\rangle
    = \langle \A_\varphibf\, \varphibf, \etabf \rangle - \re\big(\trace\big( \Sigma_\varphibf \outerprod{\varphibf}{\etabf} \big) \big)
    = \Drm \E(\varphibf)[\etabf],
    \]
    where we again used the fact that the trace of the product of the Hermitian and skew-Hermitian matrices is imaginary. Consequently, $\P_\varphibf(\varphibf)$ is the Riesz representative of the derivative of $\,\E$ with respect to the energy-adaptive metric $g_{\ea{\varphibf}{\Sigma_\varphibf}}$. Then the representation~\eqref{eq:eagrad_raw} follows immediately from Proposition~\ref{prop:projP}. 
\end{proof}

Under additional assumptions on $\outerprod{\varphibf}{\A_{\varphibf,\Sigma_\varphibf}^{-1}(\varphibf)}$, the matrix equation \eqref{eq:Xphi} can be solved explicitly, leading to a more explicit characterization of the energy-adaptive Riemannian gradient.
\begin{crl}\label{cor:ea_grad_shift}
Let Assumption~\textup{\ref{asm:shift}} be fulfilled.
If the matrix $\outerprod{\varphibf}{\A_{\varphibf,\Sigma_\varphibf}^{-1}(\varphibf)}$ is Hermitian, invertible and commutes with~$\Sigma_\varphibf$, then the solution to \eqref{eq:Xphi} has the form $X_{\varphibf} = \outerprod{\varphibf}{\A_{\varphibf,\Sigma_\varphibf}^{-1}(\varphibf)}^{-1}$, and the energy-adaptive Riemannian gradient of $\,\E$ at $\varphibf\in\St$ is given by
    \begin{equation}\label{eq:eagrad}
    \grad_{\ea{\varphibf}{\Sigma_\varphibf}} \,\E(\varphibf) =  \varphibf - \A_{\varphibf, \Sigma_\varphibf}^{-1} (\varphibf) \outerprod{\varphibf}{\A_{\varphibf,\Sigma_\varphibf}^{-1}(\varphibf)}^{-1}.
    \end{equation}
\end{crl}

\begin{proof}
Let $X_{\varphibf} = \outerprod{\varphibf}{\A_{\varphibf,\Sigma_\varphibf}^{-1}(\varphibf)}^{-1}$. The relation $\outerprod{\varphibf}{\A_{\varphibf,\Sigma_\varphibf}^{-1}(\varphibf)}\Sigma_\varphibf=\Sigma_\varphibf\outerprod{\varphibf}{\A_{\varphibf,\Sigma_\varphibf}^{-1}(\varphibf)}$ implies that $X_{\varphibf}$ also commutes with $\Sigma_\varphibf$. Consequently, for all $\vbf \in \HHH$, we have 
\[
    \A_{\varphibf, \Sigma_\varphibf}(\vbf X_{\varphibf}) 
    = \A_\varphibf( \vbf X_{\varphibf}) - \vbf X_{\varphibf} \Sigma_\varphibf  
    = (\A_\varphibf \vbf) X_{\varphibf} - \vbf\,  \Sigma_\varphibf X_{\varphibf}
    = \A_{\varphibf, \Sigma_\varphibf} (\vbf) X_{\varphibf}.
\]
Since the shifted Hamiltonian $\A_{\varphibf, \Sigma_\varphibf}$ is invertible, for arbitrary $\vbf' \in \HHH'$, this also yields 
\begin{equation}\label{eq:invH}
\A_{\varphibf, \Sigma_\varphibf}^{-1}(\vbf'  X_{\varphibf}) = \A_{\varphibf, \Sigma_\varphibf}^{-1}(\vbf') X_{\varphibf}.
\end{equation}
Therefore, 
\begin{align*}
    \outerprod{ \A_{\varphibf, \Sigma_\varphibf}^{-1}(\varphibf X_{\varphibf})}{\varphibf} + \outerprod{\varphibf}{\A_{\varphibf, \Sigma_\varphibf}^{-1}(\varphibf X_{\varphibf})} 
    & = \outerprod{ \A_{\varphibf, \Sigma_\varphibf}^{-1}(\varphibf) X_{\varphibf}}{\varphibf} + \outerprod{\varphibf}{\A_{\varphibf, \Sigma_\varphibf}^{-1}(\varphibf) X_{\varphibf}} \\ 
    & = X_{\varphibf}\outerprod{ \A_{\varphibf, \Sigma_\varphibf}^{-1}(\varphibf) }{\varphibf} + \outerprod{\varphibf}{\A_{\varphibf, \Sigma_\varphibf}^{-1}(\varphibf)} X_{\varphibf}
    \\&= 2I_p.
\end{align*}
Thus, $X_{\varphibf}$ solves equation \eqref{eq:Xphi}. 
Using \eqref{eq:eagrad_raw} and \eqref{eq:invH}, we finally get \eqref{eq:eagrad}.
\end{proof}

For a coercive Hamiltonian and $\Sigma_\varphibf = 0$, the expression~\eqref{eq:eagrad} of the energy-adaptive Riemannian gradient coincides with that obtained in~\cite{AltPS22}. A straightforward consequence of Corollary~\ref{cor:ea_grad_shift} is that the energy-adaptive Riemannian gradient is efficiently calculated using only one inversion of the shifted Hamiltonian. In the present setting, this is also true for a~uniform shift $\Sigma_\varphibf = -\sigma_\varphibf I_p$ with sufficiently large $\sigma_\varphibf$ that ensures coercivity and invertibility of~$\A_{\varphibf, \Sigma_\varphibf}$.  
Since $\A_{\varphibf, \Sigma_\varphibf}$ (and also $\A_{\varphibf, \Sigma_\varphibf}^{-1}$) acts component-wise, the matrix
$\outerprod{\varphibf}{\A_{\varphibf, \Sigma_\varphibf}^{-1}(\varphibf)}$ is Hermitian, invertible, and obviously commutes with $\Sigma_\varphibf$. Thus, the assumptions of Corollary~\ref{cor:ea_grad_shift} are fulfilled. In the next section, we will show that for more general non-uniform yet practically favorable shifts for which the assumptions of Corollary~\ref{cor:ea_grad_shift} are no longer satisfied, \eqref{eq:eagrad} still provides a sufficiently accurate approximation to the exact energy-adaptive Riemannian gradient.

\subsection{Shifting strategy}\label{sec:shift}
Although appropriate uniform shifts are not that hard to compute, they usually do not yield the best possible results in practice. Overly large shifts degrade performance, while minimal scalar shifts can render the shifted Hamiltonian ill-conditioned, leading to an~inefficient metric and poor convergence in optimization methods. Consequently, more problem-aware and matrix-valued (yet computationally inexpensive) shifts are needed to reduce the computational complexity while maintaining effective preconditioning. 

By using a shift $\Sigma_\varphibf = \Lambda_\varphibf$ with $\Lambda_\varphibf$ defined in \eqref{eq:Lambda}, the operator $\A_{\varphibf,\Sigma_\varphibf}$ can be viewed as an approximation to the Riemannian Hessian of $\,\E$ at $\varphibf \in \St$ with respect to the $\LLL$-metric $\langle \cdot, \cdot \rangle_{\LLL}$ obtained by neglecting the terms in the normal space with respect to $\langle \cdot, \cdot \rangle_{\LLL}$ and also the terms with the second-order derivative of the nonlinearity, see \cite{AltPS24} for details.
As it turns out, this shift does not generally guarantee that the resulting shifted Hamiltonian is coercive on $\Tan_\varphibf\, \St$. To remedy this shortcoming, we consider a~corrected shift 
\[
\Sigma_{\varphibf,\mu} = \Lambda_\varphibf-\mu I_\pp
\]
with a~parameter $\mu > 0$. The following proposition shows that for almost all $\mu > 0$, the shifted Hamiltonian $\A_{\varphibf,\Sigma_{\varphibf,\mu}}$ satisfies Assumption~\ref{asm:shift} as required in the previous subsection.

\begin{prop}\label{prop:shift}
Let $\varphibf_{\star} \in \St$ be a~ground state of $\,\E$ that satisfies the Aufbau principle, and let the associated Lagrange multiplier $\Lambda_{\star}$ 
have the eigenvalues \mbox{$\lambda_1\leq \ldots\leq \lambda_p$}. Then, for all $\mu > 0$, except for $\mu = \lambda_l - \lambda_m$, $1\le m < l \le p$, there exits an~open neighborhood $U_{\varphibf_\star} \subset \St$ of $\varphibf_{\star}$, where for all $\varphibf \in U_{\varphibf_{\star}}$, the shift $\Sigma_{\varphibf,\mu}=\Lambda_\varphibf-\mu I_\pp$ with $\Lambda_\varphibf$ as in~\eqref{eq:Lambda} fulfills Assumption~\textup{\ref{asm:shift}}.
\end{prop}

\begin{proof}
Let $\varphibf_{\star} \in \St$ be a ground state with the associated Lagrange multiplier $\Lambda_{\star}$ and let $\mu > 0$. We aim to prove that $\Sigma_{\varphibf_{\star},\mu}=\Lambda_{\star}-\mu I_p$ satisfies  Assumption~\ref{asm:shift}. As the Hamiltonian $\A_{\varphibf, \Sigma_{\varphibf,\mu}}$ depends smoothly on $\varphibf$, the shift $\Sigma_{\varphibf,\mu}$ also satisfies these assumptions for all $\varphibf$ in an~open neighborhood of $\varphibf_{\star}$. 

First, we show that there exists a~uniform shift of $\A_{\varphibf_{\star},\Lambda_{\star}}$ enabling the coercivity on $\HHH$. For any $\vbf\in\HHH$, we have
\[
\lambda_1\| \vbf \|_{\LLL}^2 
= \langle \lambda_1 \vbf, \vbf  \rangle_{\LLL}
\le \langle \vbf \Lambda_{\star}, \vbf  \rangle_{\LLL} 
\le \langle \lambda_p \vbf, \vbf \rangle  
= \lambda_p\|\vbf \|_{\LLL}^2.
\]
Then using the shift $\sigma_{\varphibf_{\star}}$ as in Remark~\ref{rem:Garding}, for $\tilde{\sigma}_{\varphibf_{\star}}= \sigma_{\varphibf_{\star}} + \lambda_p > 0$, we obtain that
\begin{equation}\label{eq:garding2}
    \langle \A_{\varphibf_{\star}, \Lambda_{\star}}(\vbf) , \vbf \rangle + \tilde{\sigma}_{\varphibf_{\star}}\|\vbf\|_{\LLL}^2 
    = \langle \A_{\varphibf_{\star}}\vbf,\vbf\rangle +  \sigma_{\varphibf_{\star}} \|\vbf\|_{\LLL}^2
    - (\langle \vbf\,\Lambda_{\star}, \vbf \rangle_{\LLL} - \lambda_p \|\vbf\|_{\LLL}^2) \ge \|\vbf\|_{\HHH}^2.
\end{equation}
We note that $\tilde{\sigma}_{\varphibf_{\star}} > 0$, since 
$\A_{\varphibf_{\star}, -\sigma_{\varphibf_{\star}} I_p}$ is coercive and thus only has positive eigenvalues. Consequently, 
$\sigma_{\varphibf_{\star}} + \lambda_1 > 0$, which is a~lower bound of $\tilde \sigma_{\varphibf_{\star}}$.
   
Next, we prove that $\A_{\varphibf_{\star},\Lambda_{\star}}$ is positive semidefinite on $\Tan_{\varphibf_{\star}} \St$. To this end, for any $\etabf \in \Tan_{\varphibf_{\star}} \St$, we consider the horizontal-vertical additive decomposition $\etabf = \xibf + \varphibf_{\star} M$, where $M \in \Field^{\pp \times \pp}$ is skew-Hermitian and $\xibf$ is in the horizontal space, meaning $\outerprod{\varphibf_{\star}}{\xibf} = 0_p$, see~\cite{AltPS24}. As $\varphibf_{\star}$ is a critical point of $\,\E$, it solves the NLEVP~\eqref{eq:evphk}. Then using
\[
\A_{\varphibf_{\star}, \Lambda_{\star}}(\varphibf_{\star} M) = \A_{\varphibf_{\star}} \varphibf_{\star} M - \varphibf_{\star} M \Lambda_{\star}  = \varphibf_{\star} (\Lambda_{\star} M - M \Lambda_{\star}),
\]
we calculate
\begin{align*}
    \langle \A_{\varphibf_{\star}, \Lambda_{\star}}(\xibf) , \xibf \rangle 
    &= \langle \A_{\varphibf_{\star}} \xibf, \xibf \rangle - \langle \xibf \Lambda_{\star}, \xibf \rangle_{\LLL}
    \ge (\lambda_{p+1} - \lambda_p) \|\xibf\|_{\LLL}^2 \ge 0, \\
         \langle \A_{\varphibf_{\star}, \Lambda_{\star}}(\varphibf_{\star} M) , \xibf \rangle 
         &= \langle \varphibf_{\star}(\Lambda_{\star} M \!-\! M\Lambda_{\star}), \xibf \rangle_{\LLL}
         = \re\big(\trace{\big((\Lambda_{\star} M \!-\! M\Lambda_{\star})^\ast \outerprod{\varphibf_{\star}}{\xibf}\big)}\big)\! = 0, \\
         \langle \A_{\varphibf_{\star}, \Lambda_{\star}}(\varphibf_{\star} M) , \varphibf_{\star} M \rangle 
         &=  \langle \varphibf_{\star} (\Lambda_{\star} M\! -\! M \Lambda_{\star}), \varphibf_{\star} M \rangle_{\LLL}
        = \re\big(\trace{\big((\Lambda_{\star} M\! -\! M\Lambda_{\star})^\ast M\big)}\big) = 0.
    \end{align*}
    Here, we used the properties that $\Lambda_{\star} M - M\Lambda_{\star}$ is Hermitian and that the trace of the product of the Hermitian and skew-Hermitian matrices is imaginary. Thus, we obtain that  
    \begin{align*}
        \langle \A_{\varphibf_{\star}, \Lambda_{\star}}(\etabf) , \etabf \rangle &= \langle \A_{\varphibf_{\star}, \Lambda_{\star}}(\xibf) , \xibf \rangle + 2\langle \A_{\varphibf_{\star}, \Lambda_{\star}}(\varphibf_{\star} M) , \xibf \rangle + \langle \A_{\varphibf_{\star}, \Lambda_{\star}}(\varphibf_{\star} M) , \varphibf_{\star} M \rangle \ge 0.
    \end{align*}
    
We now show that $\A_{\varphibf_{\star},\Sigma_{\varphibf_\star,\mu}}$ is coercive on $\Tan_{\varphibf_{\star}} \St$.  For $\mu \geq \tilde{\sigma}_{\varphibf_{\star}}$, the coercivity immediately follows from~\eqref{eq:garding2}. For $\mu < \tilde{\sigma}_{\varphibf_{\star}}$,  define $q \coloneqq \tilde{\sigma}_{\varphibf_{\star}}\mu^{-1} > 1$. Then for all \mbox{$\etabf \in \Tan_{\varphibf_{\star}} \St$}, we have
    \begin{align*}
        q\langle \A_{\varphibf_{\star}, \Sigma_{\varphibf_\star,\mu}} (\etabf), \etabf \rangle 
        &= q\langle \A_{\varphibf_{\star}, \Lambda_{\star}} (\etabf), \etabf \rangle + q \mu \|\etabf\|_{\LLL}^2
        > \langle \A_{\varphibf_{\star}, \Lambda_{\star}} (\etabf), \etabf \rangle + \tilde{\sigma}_{\varphibf_{\star}} \|\etabf\|_{\LLL}^2
        \ge \|\etabf\|_{\HHH}^2,
    \end{align*}
    where we used \eqref{eq:garding2} once more. Dividing by $q$ yields the coercivity of~$\A_{\varphibf_{\star},\Sigma_{\varphibf_\star,\mu}}$ on $\Tan_{\varphibf_{\star}} \St$ with the coercivity constant $q^{-1}$.
    
    In order to show the invertibility of $\A_{\varphibf_{\star},\Sigma_{\varphibf_\star,\mu}}$ on $\HHH$, we consider the eigenvalues of~$\A_{\varphibf_{\star}, \Sigma_{\varphibf_\star,\mu}}$ given by $\lambda_m - \lambda_l + \mu$ for $l = 1, \dots, p$ and $m = 1, 2, \ldots$.
    Consequently, $\A_{\varphibf_{\star}, \Sigma_{\varphibf_\star,\mu}}$ is invertible if $0<\mu \neq \lambda_l - \lambda_m$ for $1\le m<l \le p$.
\end{proof}

We see directly from the proof that without correction ($\mu = 0$), the shifted Hamiltonian $\A_{\varphibf, \Lambda_\varphibf}$ is not positive definite on the vertical space. A corrected shift $\Sigma_{\varphibf,\mu}$ with $\mu>0$ ensures the coercivity of the shifted Hamiltonian $\A_{\varphibf,\Sigma_{\varphibf,\mu}}$ on the whole tangent space $\Tan_\varphibf\,\St$, and larger values of $\mu$ grow the neighborhood of $\varphibf_\star$, where $\A_{\varphibf, \Sigma_{\varphibf, \mu}}$ is coercive on the respective tangent space. The choice of $\mu$ can thus be seen as a~trade-off between Hessian approximation (small $\mu$) and the coercivity assurance (large $\mu$). In our experiments, starting with a sufficiently good initial guess, $\mu = 0.01$ proved to be a good choice if the residual is sufficiently small. 

\begin{rem}\label{rem:inexact_calc}
    In general, the assumptions on $\outerprod{\varphibf} {\A_{\varphibf, \Sigma_{\varphibf, \mu}}^{-1}(\varphibf)}$ in Corollary~\textup{\ref{cor:ea_grad_shift}} do not hold for the corrected shift~$\Sigma_{\varphibf, \mu}$. However, if the residual is sufficiently small, then
    \begin{align*}
        \outerprod{\varphibf} {\A_{\varphibf, \Sigma_{\varphibf, \mu}}^{-1}(\varphibf)} \approx (\Lambda_\varphibf - \Sigma_{\varphibf, \mu})^{-1} = \mu^{-1} I_p
    \end{align*}
    with $\big\|\outerprod{\varphibf} {\A_{\varphibf, \Sigma_{\varphibf, \mu}}^{-1}(\varphibf)} - \mu^{-1}I_p\big\|_F = \mathcal O(\|\mathrm{res}\,\E(\varphibf)\|_{\HHH'}^2)$, where $\|\cdot\|_F$ denotes the Frobenius matrix norm and $\|\cdot\|_{\HHH'}$ is the dual norm. As~$\mu^{-1} I_p$ is Hermitian, invertible and commutes with~$\Sigma_{\varphibf,\mu}$, we expect the error resulting from the approximation of the energy-adaptive Riemannian gradient by \eqref{eq:eagrad} to be of the order of magnitude of $\|\mathrm{res}\,\E(\varphibf)\|_{\HHH'}^2$. Thus, this approximation is sufficiently accurate, when the relative tolerance for inverting the shifted Hamiltonian is in the range of $\|\mathrm{res}\,\E(\varphibf)\|_{\HHH'}^2$.
\end{rem}

Choosing the energy-adaptive metric \(g_{\ea{\varphibf}{\Sigma_{\varphibf,\mu}}}\) in Algorithm~\ref{alg:rcg} and using the right-hand side of \eqref{eq:eagrad} to approximately compute the energy-adaptive Riemannian gradient, we obtain the EARCG method to solve the Kohn-Sham energy minimization problem \eqref{eq:prob_e_min}. This method is presented in Algorithm~\ref{alg:earcg}.

\begin{algorithm2e}[th]
\SetKwInOut{Input}{input}
\SetCustomAlgoRuledWidth{\textwidth}
\SetAlgoLined
\Input{initial guess $\varphibf^{(0)} \in \St$, retraction $\R_\varphibf$, vector transport  $\T^\R$, shift~correction~$\mu > 0$, initial step size $\tau^{(-1)}=1$}
\BlankLine
$\Lambda^{(0)} = \coouterprod{\A_{\varphibf^{(0)}} \varphibf^{(0)}}{\varphibf^{(0)}}$\;
$\rbf^{(0)} = \A_{\varphibf^{(0)}} \varphibf^{(0)} - \varphibf^{(0)} \Lambda^{(0)}$  \tcc*{residual}  
$\Sigma_\mu^{(0)} = \Lambda^{(0)} - \mu I_p$  \tcc*{shift}
solve $\A_{\varphibf^{(0)}, \Sigma^{(0)}_\mu}(\xbf^{(0)})=\varphibf^{(0)}$ for $\xbf^{(0)}$\;
$\gbf^{(0)}= \varphibf^{(0)} - \xbf^{(0)}\outerprod{\varphibf^{(0)}}{\xbf^{(0)}}^{-1}$ \;
$\etabf^{(0)} = - \gbf^{(0)}$  \tcc*{search direction} 
 \For{$m=0, 1, 2, \dots$ until convergence}{
 compute the step size $\tau^{(m)}$ using Algorithm~\ref{alg:mod_secant_step} with the initial guess $\tau^{(m-1)}$ \;
    $\varphibf^{(m+1)} = \R^{}_{\varphibf^{(m)}}(\tau^{(m)} \etabf^{(m)})$\;
    $\Lambda^{(m+1)} = \coouterprod{\A_{\varphibf^{(m+1)}} \varphibf^{(m+1)}}{\varphibf^{(m+1)}}$\;
    $\rbf^{(m+1)} = \A_{\varphibf^{(m+1)}} \varphibf^{(m+1)} - \varphibf^{(m+1)} \Lambda^{(m+1)}$ \; 
    $\Sigma_\mu^{(m+1)} = \Lambda^{(m+1)} - \mu I_p $\;
    solve $\A_{\varphibf^{(m+1)}, \Sigma^{(m+1)}_\mu}(\xbf^{(m+1)})=\varphibf^{(m+1)}$ for $\xbf^{(m+1)}$\;
    $\gbf^{(m+1)}= \varphibf^{(m+1)} - \xbf^{(m+1)}\outerprod{\varphibf^{(m+1)}}{\xbf^{(m+1)}}^{-1}$ \;
    $\beta^{(m+1)} = \max\bigg\{0, \min\Big\{\frac{\langle \rbf^{(m+1)},\, \gbf^{(m+1)} \rangle}{\langle \rbf^{(m)},\, \gbf^{(m)} \rangle}, \frac{\langle \rbf^{(m+1)},\, \gbf^{(m+1)} - \T^{\R}_{\tau^{(m)} \eta^{(m)}} (\gbf^{(m)}) \rangle_{}}{\langle \rbf^{(m)},\, \gbf^{(m)} \rangle_{}}\Big\}\bigg\}$\;
    $\etabf^{(m+1)} = - \gbf^{(m+1)} + \beta^{(m+1)} \T^{\R}_{\tau^{(m)} \etabf^{(m)}} (\etabf^{(m)})$\; 
 }
 \KwRet{$\varphibf^{(m+1)}$}
 \caption{\scalebox{1}{Energy-adaptive RCG for the Kohn-Sham energy minimization problem}\!\!\!\!\!\!}
     \label{alg:earcg}
\end{algorithm2e}

In this algorithm, the step size strategy and the calculation of CG parameters make heavy use of \eqref{eq:calc_metric_g}. The computation of the residual $\rbf=\mathrm{res}\,\E(\varphibf)$ and the duality pair $\langle \cdot, \cdot \rangle$ involves one application of the Hamiltonian, which is much cheaper than calculating the Riemannian gradient in the energy-adaptive metric and then applying this metric, which requires one application and one inversion of the shifted Hamiltonian operator.

\section{Numerical experiments}\label{sec:experiments} 
The presented algorithms have been implemented in Julia using the Density Functional Toolkit \texttt{DFTK.jl}\footnote{\texttt{https://docs.dftk.org}} \cite{DFTKjcon}. The source code is available at \begin{center}
\texttt{https://github.com/jonas-pueschel/RCG\_DFTK}
\end{center}
We first briefly describe the discretization used in \texttt{DFTK}, then discuss the choice of parameters for EARCG and the other state-of-the-art methods tested, and finally compare them in numerical experiments.

\subsection{Plane-wave discretization} \label{sec:PlaneWave}
Let us start with a brief overview of the discretization used for periodic quantum systems, as implemented in \texttt{DFTK}, among others. For more details, the interested reader is referred to \cite[Chapter III]{Lev20}. For the spatial discretization of the Kohn-Sham model, the plane-wave discretization method with an $L^2$-orthonormal Fourier basis is used, where the number of degrees of freedom $n$ is determined by a cutoff for the kinetic energy $E_{\rm cut}>0$. 
A $\pp$-frame $\vbf$ is then represented by a~matrix $V \in \Field^{n\times \pp}$.
Due to the orthonormality of the Fourier basis, the (complex)
outer product of $V, W \in \Field^{n\times \pp}$ is given by $V^\ast W \in \Field ^{\pp\times \pp}$,
and the (real) inner product coincides with the (real) Frobenius inner product $\langle V,W \rangle_F = \re\big(\mathrm{tr} (V^\ast  W)\big)$ which induces the Frobenius matrix norm $\|V\|_F=\sqrt{\langle V,V \rangle_F}$.
Thus, the discrete version of the Stiefel manifold reads 
\[
\mathrm{St}(\pp,n) = \big\{\varPhi\in\Field^{n\times \pp}\enskip :\enskip \varPhi^\ast \varPhi=I_p\big\}.
\]
Since all plane-wave basis functions are eigenfunctions of the Laplacian $\Delta$, this discretization allows us to use the (shifted) kinetic energy operator $-\tfrac{1}{2}\Delta + \alpha \I$ with the canonical identification $\I:\HHH\hookrightarrow\HHH'$ and $\alpha>0$ as a~preconditioner, since it is a diagonal matrix in the Fourier basis. It is commonly referred to as the Teter-Payne-Allan (TPA) preconditioner \cite{TPA89} and is the standard preconditioner used for plane-wave DFT~\cite{CiCP-18-167}. 

When considering the periodic structures, a $k$-point grid of size $k_1\times k_2\times k_3$ is used. This means that a virtual $k_1\times k_2\times k_3$ supercell of the lattice is constructed and solutions invariant under different grid-translation operators are considered separately.
The total number of $k$-points~$K$ is $k_1k_2k_3$ divided by the symmetries of the lattice. Consequently, the discrete state admits the form $\varPhi=[\varPhi_1^\ast, \ldots,\varPhi_K^\ast  ]^\ast  $ with $\varPhi_j\in \Field^{n_j\times p}$ for $j=1,\ldots, K$, where~$n_j$ is the number of degrees of freedom at the $j$-th $k$-point. The total number of degrees of freedom is $n=n_1+\ldots+n_k$. The discrete Kohn-Sham energy minimization problem is then formulated on the product manifold 
\begin{equation}\label{eq:ProdMan}
\mathrm{St}(\pp,n_1) \times \dots \times \mathrm{St}(\pp,n_K)
\end{equation}
for which all geometric concepts can be derived from those on the Stiefel manifolds $\mathrm{St}(\pp,n_j)$; see \cite{GaoPY23}. 
Furthermore, it follows from Bloch's theorem \cite{Bloch1929} that the discrete Hamiltonian $H_\varPhi$ has the block diagonal structure with diagonal blocks corresponding to $k$-points. Thus, the energy-adaptive metric can be defined as a~sum of the $k$-point-wise energy-adaptive metrics using the corresponding Hamiltonian blocks. In this case, the theory for the Stiefel manifold $\St$ discussed in the sections above can be adapted to the product manifold~\eqref{eq:ProdMan}.

\subsection{Experimental setup}
Here we outline how the RCG schemes and other methods are explicitly set up for subsequent numerical experiments. Unless otherwise specified, our choice of the retraction, vector transport, CG parameters, and step size for RCG are consistent with the corresponding suggestions in Section~\ref{sec:rcg}. For the step size parameters we explicitly choose
\[
\delta = 0.05, \quad\sigma = 0.1, \quad\gamma = 0.5, \quad\epsilon = 10^{-12}.
\]
Since the scaling of the vector transport to prevent length growth involves an additional application of the Hamiltonian operator, we omit it for the sake of efficiency. In our numerical experiments, we did not observe any adverse effects due to the possible length increase of the transported direction.
     
The initial guess $\varPhi^{(0)}$ is the same for all methods and is generated by running SCF until the density difference satisfies $\|\rho^{(j+1)} - \rho^{(j)}\|_2 < 0.1$. 
For the energy-adaptive metric, we use the shift matrix $\Sigma_{\varPhi, \mu}=\varPhi^*H_\varPhi\,\varPhi -\mu I_p$ (cf.~Proposition~\ref{prop:shift}), where we choose the correction parameter $\mu = 0.01$. To compute the inexact energy-adaptive Riemannian gradient, we use the expression \eqref{eq:eagrad} whose discrete counterpart is
\[
G = \varPhi - X (\varPhi^\ast X)^{-1}, 
\]
where $X \in \Field ^{n\times p}$ is the solution to the matrix Sylvester equation
\begin{equation} \label{eq:Sylv}
H_\varPhi X - X \Sigma_{\varPhi,\mu} = \varPhi.
\end{equation}
This equation is solved approximately using the preconditioned block full orthogonalization method (FOM) from Algorithm~\ref{alg:pfom} given in Appendix~\ref{app:FOM}. The FOM iteration is terminated when the relative residual für \eqref{eq:Sylv} satisfies the condition
\[
\frac{\|H_\varPhi X_j - X_j \Sigma_{\varPhi,\mu}  - \varPhi\|_F}{\|H_\varPhi X_0 - X_0 \Sigma_{\varPhi,\mu} - \varPhi\|_F} \le 
2.5 \cdot 10^{-2},
\]
where the initial guess is given by $X_0 = \varPhi(\varPhi^*H_\varPhi\,\varPhi - \Sigma_{\varPhi,\mu})^{-1}=\mu^{-1} \varPhi$, a modified version of \cite{AltPS22} that accounts for the shift. Since in our experiments, $\varPhi^{(0)}$ generated by the strategy discussed above is sufficiently close to a local minimizer, this tolerance fulfills the condition in \mbox{Remark}~\ref{rem:inexact_calc}.
As a~preconditioner, we use the TPA preconditioner \cite{TPA89} as introduced in Section~\ref{sec:PlaneWave}, where $\alpha$ is chosen as the orbital-wise mean kinetic energy, an idea from \cite[Section IV]{TPA89}, which is already implemented in \texttt{DFTK}.

The following RCG and RGD variants have been tested:
\begin{itemize}    
    \item EARCG\quad energy-adaptive RCG (Algorithm~\ref{alg:earcg});
    \item EARGD\quad energy-adaptive RGD with the Barzilai-Borwein step size and non-mono\-to\-ne backtracking as in \cite[Algorithm~2]{AltPS22};
    \item L2RCG\quad (unpreconditioned) $L^2$-gradient RCG with the approximated Hessian step size \cite{DaiLZZ17}, comparable to the methods given in \cite{DaiLZZ17, LuoWR24, Oviedo2018}.
\end{itemize}
Moreover, we compared them to the following state-of-the-art methods implemented in \texttt{DFTK}:
\begin{itemize}
    \item SCF\quad self-consistent field iteration with (optimized) default parameters including LDOS based mixing, Anderson acceleration, and adaptive bands;
    \item LBFGS\quad direct constraint minimization utilizing \texttt{LBFGS} from \texttt{Optim.jl} \cite{Optim} and the Hager-Zhang step size implemented in \texttt{LineSearches.jl} and using the TPA-precon\-di\-tio\-ned residual.
\end{itemize}

\subsection{Performance comparison}
Comparing different methods in the field of electronic structure calculations is notoriously difficult \cite{CKL22}, since these methods may have different criteria governing their convergence behavior. Therefore, our goal is not to give a comprehensive performance evaluation of EARCG, but to find framework conditions under which it may be viable compared to already established methods. Thus, we limit ourselves to a set of rather small lattice solid-state materials and consider the number of applications of the Hamiltonian operator and CPU time as cost measures. We find that for smaller lattices with only a few electrons, our method performs very well. For systems with more electrons and larger lattices, we expect the performance advantage to decrease due to long-range Coulomb interactions that degrade the quality of the descent direction. In addition, the (inaccurate) calculation of the energy-adaptive Riemannian gradient becomes more expensive as the number of electrons increases.

We compare the performance of the different methods by applying them to three solid-state materials for which we perform zero-temperature ground state calculations:
    \begin{itemize}
    \item Silicon standard crystalline lattice, LDA model, $6\times6\times 6$ $k$-point grid, $E_{cut}=50\,\mathrm{Ha}$, lattice constant $10.26$~Bohrs, $n=72912$, $p = 4$, HOMO-LUMO gap $0.02310$;
    \item GaAs periodic lattice for the FCC phase, LDA model, $2\times2\times2$ $k$-point grid, \linebreak \mbox{$E_{cut}=60\,\mathrm{Ha}$}, lattice constant $10.68$ Bohrs, $n = 20295$, $p = 4$, HOMO-LUMO gap $0.01315$;
    \item TiO$_2$ MP-2657 configuration in the primitive cell, LDA model, $2\times2\times2$ $k$-point grid, $E_{cut}=60\,\mathrm{Ha}$, $n=57893$, $p = 16$, HOMO-LUMO gap $0.06398$.
\end{itemize}
The silicon model, taken from the \texttt{DFTK} documentation \cite{DFTKjcon}, is a well-established benchmark in semiconductor simulation and provides a reliable baseline for performance comparisons. The GaAs and TiO\textsubscript{2} models, taken from the implementation in \cite{CDK21},
are more demanding with increased computational complexity and richer electronic structure.

\begin{figure}
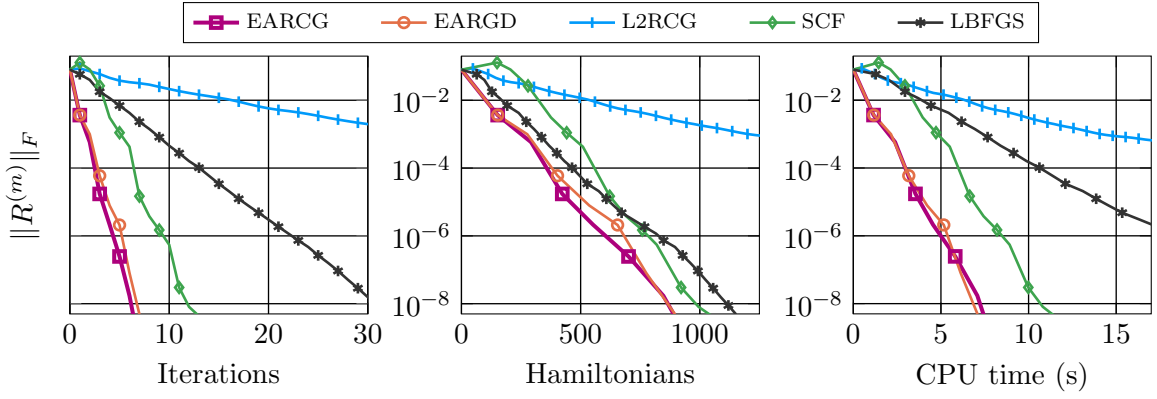
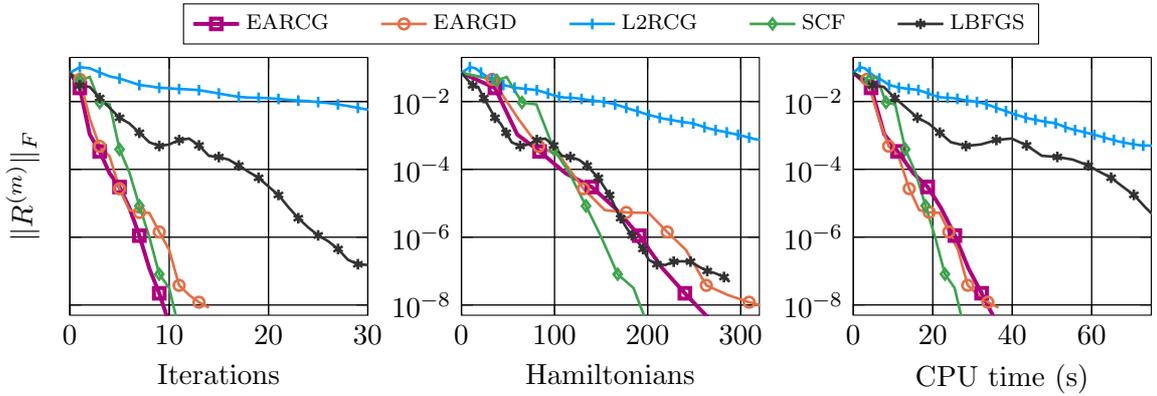
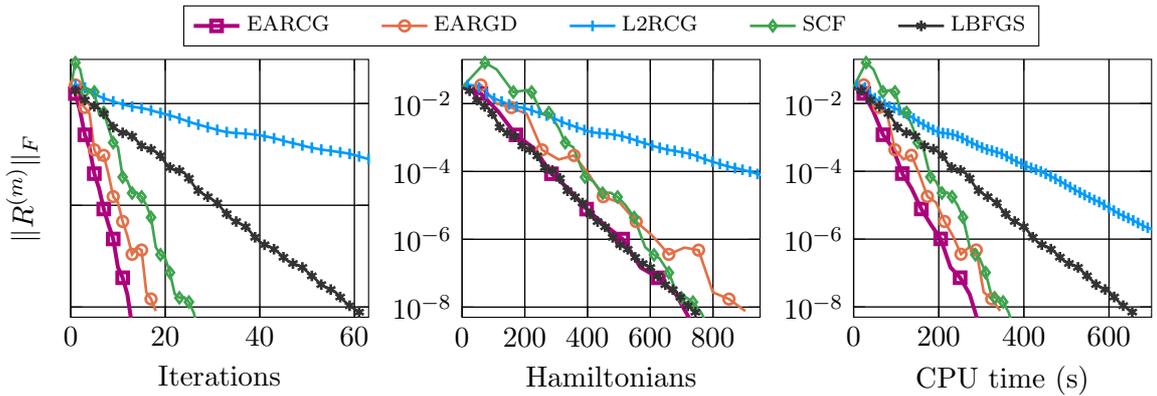

\begin{subfigure}{\textwidth}
\centering
    \begin{tikzpicture}[/tikz/background rectangle/.style={fill={rgb,1:red,1.0;green,1.0;blue,1.0}, fill opacity={1.0}, draw opacity={1.0}}, show background rectangle]

\begin{axis}[
name = silicon1,
at={(0.0mm,0.0mm)},
width={\plotwidth}, height={\plotheight}, 
xlabel={Iterations}, 
x tick style={color={rgb,1:red,0.0;green,0.0;blue,0.0}, opacity={1.0}}, 
xmin={0}, xmax={30}, 
xtick align={inside}, 
xmajorgrids={true}, 
x grid style={color={rgb,1:red,0.0;green,0.0;blue,0.0}, draw opacity={0.1}, line width={0.5}, solid}, 
scaled y ticks={false}, 
ylabel={$\|R^{(m)}\|_F$}, 
ylabel style={font={{\fontsize{11 pt}{14.3 pt}\selectfont}}, color={rgb,1:red,0.0;green,0.0;blue,0.0}, draw opacity={1.0}, rotate={0.0}}, 
ymode={log}, log basis y={10}, 
ymin={5e-9}, ymax={0.2}, 
yticklabels=\empty, 
ytick={{1.0e-8,1.0e-6,0.0001,0.01}}, 
ytick align={inside}, 
yticklabel style={font={{\fontsize{8 pt}{10.4 pt}\selectfont}}, color={rgb,1:red,0.0;green,0.0;blue,0.0}, draw opacity={1.0}, rotate={0.0}}, 
ymajorgrids={true},
y grid style={color={rgb,1:red,0.0;green,0.0;blue,0.0}, draw opacity={0.1}, line width={0.5}, solid}, 
colorbar={false}
]

\addplot[color=unia-purple, draw opacity={1.0}, line width={1.5}, solid, mark=square, mark repeat=2,mark phase=2]
table[row sep={\\}]
{
    \\
    0 0.07952763828001923 \\
1 0.003640662762086863 \\
2 0.000581085190163284 \\
3 1.7274551062258736e-5 \\
4 2.1426977415352274e-6 \\
5 2.4725573628428215e-7 \\
6 1.720944897604908e-8 \\
7 5.460654298077895e-10 \\
};

\addplot[color=cpl1, draw opacity={1.0}, line width={1}, solid, mark=o, mark repeat=2,mark phase=2]
table[row sep={\\}]
{
    \\
    0 0.07952763828001923 \\
1 0.003640662762086863 \\
2 0.0009963442919841965 \\
3 5.94749658471245e-5 \\
4 7.896543742392376e-6 \\
5 2.09488566270105e-6 \\
6 8.088654121203195e-8 \\
7 4.894990587913975e-9 \\
};

    \addplot[color=cpl2, draw opacity={1.0}, line width={1}, solid, mark=|, mark repeat=2, mark phase=2]
table[row sep={\\}]
      {
    \\
    0 0.07952763828001923 \\
1 0.08692498693968086 \\
2 0.0717820452237402 \\
3 0.05923506635383784 \\
4 0.04568088364542617 \\
5 0.037876706561778724 \\
6 0.03403750128353163 \\
7 0.03175569126264674 \\
8 0.028856473192861446 \\
9 0.02478629669947618 \\
10 0.02139382587341938 \\
11 0.018547730622794464 \\
12 0.016355641841361162 \\
13 0.01457968032720044 \\
14 0.013156095936498073 \\
15 0.011806136943779659 \\
16 0.010516765732818657 \\
17 0.009068800529681504 \\
18 0.007678115610052347 \\
19 0.0065720757049602995 \\
20 0.005761981138397605 \\
21 0.005198642692827996 \\
22 0.004797646194898956 \\
23 0.004373982843415201 \\
24 0.003917324558918743 \\
25 0.003469697284755493 \\
26 0.0030329596817565163 \\
27 0.0026905604886783836 \\
28 0.0024067473081049538 \\
29 0.0021667377031614403 \\
30 0.001977807200160992 \\
31 0.001788141079149819 \\
32 0.0016386920712307478 \\
33 0.0015098267597637605 \\
34 0.0013631774036882858 \\
35 0.0012274405099960074 \\
36 0.001107728308850681 \\
37 0.001009907187053467 \\
38 0.0009414032232838619 \\
39 0.0008896781857227639 \\
40 0.000848003083252443 \\
41 0.0008189791263007692 \\
42 0.0007801739105197104 \\
43 0.0007347632979382547 \\
44 0.0006907706912073304 \\
45 0.0006553451865587153 \\
46 0.0006172023317776181 \\
47 0.000583119917488437 \\
48 0.0005506190722714602 \\
49 0.0005087748608089967 \\
50 0.00046550057485192494 \\
51 0.00041625173909033057 \\
52 0.00037612461637198685 \\
53 0.0003404890901074389 \\
54 0.0003116239494612655 \\
55 0.00028459625127900696 \\
56 0.00026050484711647446 \\
57 0.00023925203777669894 \\
58 0.00021504972436908631 \\
59 0.0001929232750592563 \\
60 0.00017191214036453822 \\
61 0.0001539199905697635 \\
62 0.0001394210052777111 \\
63 0.00012447804781728054 \\
64 0.00011212023356855171 \\
65 0.00010111624205482681 \\
66 9.159436668269509e-5 \\
67 8.315159911794635e-5 \\
68 7.568068513103377e-5 \\
69 6.91803791065654e-5 \\
70 6.274615909435801e-5 \\
71 5.672641264041971e-5 \\
72 5.213093726056359e-5 \\
73 4.900397457591007e-5 \\
74 4.631368331908214e-5 \\
75 4.388550590231333e-5 \\
76 4.059820221843185e-5 \\
77 3.7513847584416435e-5 \\
78 3.52230614189692e-5 \\
79 3.2728384197680434e-5 \\
80 3.110628390624551e-5 \\
81 2.893320320737647e-5 \\
82 2.7032008576532294e-5 \\
83 2.5142326335536702e-5 \\
84 2.3207619373365465e-5 \\
85 2.1536498090057516e-5 \\
86 2.0151022075160345e-5 \\
87 1.8649318044088496e-5 \\
88 1.744165380442338e-5 \\
89 1.6216129087192333e-5 \\
90 1.4957607830768397e-5 \\
91 1.3529304506863862e-5 \\
92 1.2377570838467982e-5 \\
93 1.1260865310442447e-5 \\
94 1.0114031309798533e-5 \\
95 9.152729152529875e-6 \\
96 8.088924868743557e-6 \\
97 7.476138922229348e-6 \\
98 6.948221192284723e-6 \\
99 6.32207151309887e-6 \\
100 5.810182666295277e-6 \\
};

    \addplot[color=cpl3, draw opacity={1.0}, line width={1}, solid, mark=diamond, mark repeat=2, mark phase=2]
table[row sep={\\}]
{
    \\
    0 0.07952763828001923 \\
1 0.12893725317567875 \\
2 0.08054244118795088 \\
3 0.026600558687113734 \\
4 0.0030954432418888733 \\
5 0.0011033984790520708 \\
6 0.000425203289829024 \\
7 1.4815006321120977e-5 \\
8 3.7793131707142958e-6 \\
9 1.4887599971880954e-6 \\
10 5.474512799013683e-7 \\
11 3.007635431783204e-8 \\
12 8.068317825399006e-9 \\
13 4.304868154996554e-9 \\
};

    \addplot[color=cpl4, draw opacity={1.0}, line width={1}, solid, mark=asterisk, mark repeat=2, mark phase=2]
table[row sep={\\}]
{
    \\0 0.07952763828001923 \\
1 0.058329767116307926 \\
2 0.039798892409636244 \\
3 0.017805048515566677 \\
4 0.011528001940082815 \\
5 0.006948851130300159 \\
6 0.004233476436237086 \\
7 0.0023182806030574835 \\
8 0.0014205282915731611 \\
9 0.0008150285950397912 \\
10 0.0004548815823840144 \\
11 0.00027603513489712005 \\
12 0.00016030650103703606 \\
13 0.00010244497311644544 \\
14 5.925955266030251e-5 \\
15 3.4340292534034424e-5 \\
16 2.0877631354642103e-5 \\
17 1.2385292559899828e-5 \\
18 7.651625450363235e-6 \\
19 4.802576794391758e-6 \\
20 3.007914473851147e-6 \\
21 1.8628917432171206e-6 \\
22 1.171449846388138e-6 \\
23 7.354057375278773e-7 \\
24 4.699879769780677e-7 \\
25 2.661941890748651e-7 \\
26 1.5985203908809032e-7 \\
27 9.234193837237451e-8 \\
28 5.07289986634458e-8 \\
29 2.869390629966284e-8 \\
30 1.5604452373072072e-8 \\
31 8.888796301223622e-9 \\
32 4.866734199222215e-9 \\
};
\end{axis}


\begin{axis}[
name = silicon2,
at={(51.5mm,0.0mm)}, 
width={\plotwidth}, height={\plotheight},
xlabel={Hamiltonians}, 
xmin={0}, xmax={1250}, 
xticklabels={{$0$,$500$,$1000$}}, 
xtick={{0,500,1000}}, 
xtick align={inside}, 
x tick style={color={rgb,1:red,0.0;green,0.0;blue,0.0}, opacity={1.0}}, 
xmajorgrids={true}, 
x grid style={color={rgb,1:red,0.0;green,0.0;blue,0.0}, draw opacity={0.1}, line width={0.5}, solid}, 
scaled y ticks={false}, 
ymode={log}, log basis y={10}, ymajorgrids={true}, 
ymin={5e-9}, ymax={0.2},  ytick={{1.0e-8,1.0e-6,0.0001,0.01}}, 
ytick align={inside}, 
yticklabels={{$10^{-8}$,$10^{-6}$,$10^{-4}$,$10^{-2}$}}, 
ytick={{1.0e-8,1.0e-6,0.0001,0.01}}, 
y tick style={color={rgb,1:red,0.0;green,0.0;blue,0.0}, opacity={1.0}}, 
y grid style={color={rgb,1:red,0.0;green,0.0;blue,0.0}, draw opacity={0.1}, line width={0.5}, solid}, 
colorbar={false}
]
\addplot[color=unia-purple, draw opacity={1.0}, line width={1.5}, solid, mark=square, mark repeat=2,mark phase=2]
table[row sep={\\}]
{
    \\
    0.0 0.07952763828001923 \\
152.0 0.003640662762086863 \\
292.0 0.000581085190163284 \\
423.0 1.7274551062258736e-5 \\
552.0 2.1426977415352274e-6 \\
701.0 2.4725573628428215e-7 \\
847.0 1.720944897604908e-8 \\
967.0 5.460654298077895e-10 \\
};

\addplot[color=cpl1, draw opacity={1.0}, line width={1}, solid, mark=o, mark repeat=2,mark phase=2]
table[row sep={\\}]
{
    \\
    0.0 0.07952763828001923 \\
152.0 0.003640662762086863 \\
276.0 0.0009963442919841965 \\
403.0 5.94749658471245e-5 \\
536.0 7.896543742392376e-6 \\
655.0 2.09488566270105e-6 \\
775.0 8.088654121203195e-8 \\
896.0 4.894990587913975e-9 \\
};

    \addplot[color=cpl2, draw opacity={1.0}, line width={1}, solid, mark=|, mark repeat=2, mark phase=2]
table[row sep={\\}]
      {
    \\
    0.0 0.07952763828001923 \\
48.0 0.08692498693968086 \\
80.0 0.0717820452237402 \\
112.0 0.05923506635383784 \\
144.0 0.04568088364542617 \\
176.0 0.037876706561778724 \\
208.0 0.03403750128353163 \\
240.0 0.03175569126264674 \\
272.0 0.028856473192861446 \\
304.0 0.02478629669947618 \\
336.0 0.02139382587341938 \\
368.0 0.018547730622794464 \\
400.0 0.016355641841361162 \\
432.0 0.01457968032720044 \\
464.0 0.013156095936498073 \\
496.0 0.011806136943779659 \\
528.0 0.010516765732818657 \\
560.0 0.009068800529681504 \\
592.0 0.007678115610052347 \\
624.0 0.0065720757049602995 \\
656.0 0.005761981138397605 \\
688.0 0.005198642692827996 \\
720.0 0.004797646194898956 \\
752.0 0.004373982843415201 \\
784.0 0.003917324558918743 \\
816.0 0.003469697284755493 \\
848.0 0.0030329596817565163 \\
880.0 0.0026905604886783836 \\
912.0 0.0024067473081049538 \\
944.0 0.0021667377031614403 \\
976.0 0.001977807200160992 \\
1008.0 0.001788141079149819 \\
1040.0 0.0016386920712307478 \\
1072.0 0.0015098267597637605 \\
1104.0 0.0013631774036882858 \\
1136.0 0.0012274405099960074 \\
1168.0 0.001107728308850681 \\
1200.0 0.001009907187053467 \\
1232.0 0.0009414032232838619 \\
1264.0 0.0008896781857227639 \\
1296.0 0.000848003083252443 \\
1328.0 0.0008189791263007692 \\
1360.0 0.0007801739105197104 \\
1392.0 0.0007347632979382547 \\
1424.0 0.0006907706912073304 \\
1456.0 0.0006553451865587153 \\
1488.0 0.0006172023317776181 \\
1520.0 0.000583119917488437 \\
1552.0 0.0005506190722714602 \\
1584.0 0.0005087748608089967 \\
1616.0 0.00046550057485192494 \\
1648.0 0.00041625173909033057 \\
1680.0 0.00037612461637198685 \\
1712.0 0.0003404890901074389 \\
1744.0 0.0003116239494612655 \\
1776.0 0.00028459625127900696 \\
1808.0 0.00026050484711647446 \\
1840.0 0.00023925203777669894 \\
1872.0 0.00021504972436908631 \\
1904.0 0.0001929232750592563 \\
1936.0 0.00017191214036453822 \\
1968.0 0.0001539199905697635 \\
2000.0 0.0001394210052777111 \\
2032.0 0.00012447804781728054 \\
2064.0 0.00011212023356855171 \\
2096.0 0.00010111624205482681 \\
2128.0 9.159436668269509e-5 \\
2160.0 8.315159911794635e-5 \\
2192.0 7.568068513103377e-5 \\
2224.0 6.91803791065654e-5 \\
2256.0 6.274615909435801e-5 \\
2288.0 5.672641264041971e-5 \\
2320.0 5.213093726056359e-5 \\
2352.0 4.900397457591007e-5 \\
2384.0 4.631368331908214e-5 \\
2416.0 4.388550590231333e-5 \\
2448.0 4.059820221843185e-5 \\
2480.0 3.7513847584416435e-5 \\
2512.0 3.52230614189692e-5 \\
2544.0 3.2728384197680434e-5 \\
2576.0 3.110628390624551e-5 \\
2608.0 2.893320320737647e-5 \\
2640.0 2.7032008576532294e-5 \\
2672.0 2.5142326335536702e-5 \\
2704.0 2.3207619373365465e-5 \\
2736.0 2.1536498090057516e-5 \\
2768.0 2.0151022075160345e-5 \\
2800.0 1.8649318044088496e-5 \\
2832.0 1.744165380442338e-5 \\
2864.0 1.6216129087192333e-5 \\
2896.0 1.4957607830768397e-5 \\
2928.0 1.3529304506863862e-5 \\
2960.0 1.2377570838467982e-5 \\
2992.0 1.1260865310442447e-5 \\
3024.0 1.0114031309798533e-5 \\
3056.0 9.152729152529875e-6 \\
3088.0 8.088924868743557e-6 \\
3120.0 7.476138922229348e-6 \\
3152.0 6.948221192284723e-6 \\
3184.0 6.32207151309887e-6 \\
3216.0 5.810182666295277e-6 \\
};

    \addplot[color=cpl3, draw opacity={1.0}, line width={1}, solid, mark=diamond, mark repeat=2, mark phase=2]
table[row sep={\\}]
{
    \\
    0.0 0.07952763828001923 \\
149.5 0.12893725317567875 \\
205.5 0.08054244118795088 \\
279.25 0.026600558687113734 \\
375.75 0.0030954432418888733 \\
441.5 0.0011033984790520708 \\
509.5 0.000425203289829024 \\
622.0 1.4815006321120977e-5 \\
691.5 3.7793131707142958e-6 \\
758.75 1.4887599971880954e-6 \\
819.25 5.474512799013683e-7 \\
921.5 3.007635431783204e-8 \\
995.0 8.068317825399006e-9 \\
1051.0 4.304868154996554e-9 \\
};

    \addplot[color=cpl4, draw opacity={1.0}, line width={1}, solid, mark=asterisk, mark repeat=2, mark phase=2]
table[row sep={\\}]
{
    \\0.0 0.07952763828001923 \\
64.0 0.058329767116307926 \\
96.0 0.039798892409636244 \\
128.0 0.017805048515566677 \\
160.0 0.011528001940082815 \\
192.0 0.006948851130300159 \\
240.0 0.004233476436237086 \\
272.0 0.0023182806030574835 \\
304.0 0.0014205282915731611 \\
336.0 0.0008150285950397912 \\
368.0 0.0004548815823840144 \\
400.0 0.00027603513489712005 \\
432.0 0.00016030650103703606 \\
464.0 0.00010244497311644544 \\
496.0 5.925955266030251e-5 \\
528.0 3.4340292534034424e-5 \\
576.0 2.0877631354642103e-5 \\
608.0 1.2385292559899828e-5 \\
640.0 7.651625450363235e-6 \\
672.0 4.802576794391758e-6 \\
720.0 3.007914473851147e-6 \\
768.0 1.8628917432171206e-6 \\
816.0 1.171449846388138e-6 \\
848.0 7.354057375278773e-7 \\
896.0 4.699879769780677e-7 \\
928.0 2.661941890748651e-7 \\
960.0 1.5985203908809032e-7 \\
992.0 9.234193837237451e-8 \\
1024.0 5.07289986634458e-8 \\
1056.0 2.869390629966284e-8 \\
1088.0 1.5604452373072072e-8 \\
1120.0 8.888796301223622e-9 \\
1152.0 4.866734199222215e-9 \\
};
\end{axis}


\begin{axis}[
name = silicon3, 
at={(103.0mm,0.0mm)}, 
width={\plotwidth}, height={\plotheight},
legend cell align={left}, legend columns={5},
legend style={color={rgb,1:red,0.0;green,0.0;blue,0.0}, draw opacity={1.0}, 
solid, fill={rgb,1:red,1.0;green,1.0;blue,1.0}, fill opacity={1.0}, text opacity={1.0}, font={{\fontsize{8 pt}{10.4 pt}\selectfont}}, text={rgb,1:red,0.0;green,0.0;blue,0.0}, cells={anchor={center}}, at={(-2.25, 1.05)}, anchor={south west}},
xlabel={CPU time (s)}, 
xmin={0}, xmax={17}, xtick align={inside}, 
x tick style={color={rgb,1:red,0.0;green,0.0;blue,0.0}, opacity={1.0}}, 
xmajorgrids={true}, 
x grid style={color={rgb,1:red,0.0;green,0.0;blue,0.0}, draw opacity={0.1}, line width={0.5}, solid}, 
scaled y ticks={false}, 
ymode={log}, log basis y={10}, ymin={5e-9}, ymax={0.2}, 
y tick style={color={rgb,1:red,0.0;green,0.0;blue,0.0}, opacity={1.0}}, 
yticklabels={{$10^{-8}$,$10^{-6}$,$10^{-4}$,$10^{-2}$}}, 
ytick={{1.0e-8,1.0e-6,0.0001,0.01}}, 
ytick align={inside}, 
ymajorgrids={true},
y grid style={color={rgb,1:red,0.0;green,0.0;blue,0.0}, draw opacity={0.1}, line width={0.5}, solid}, 
colorbar={false}
]
\addplot[color=unia-purple, draw opacity={1.0}, line width={1.5}, solid, mark=square, mark repeat=2,mark phase=2]
table[row sep={\\}]
{
    \\
    0.0 0.07952763828001923 \\
1.15425931 0.003640662762086863 \\
2.384392548 0.000581085190163284 \\
3.5463923480000004 1.7274551062258736e-5 \\
4.544797846000001 2.1426977415352274e-6 \\
5.814445795 2.4725573628428215e-7 \\
7.075908412 1.720944897604908e-8 \\
8.033343654000001 5.460654298077895e-10 \\
};
\addlegendentry {EARCG$ \qquad$}

\addplot[color=cpl1, draw opacity={1.0}, line width={1}, solid, mark=o, mark repeat=2,mark phase=2]
table[row sep={\\}]
{
    \\
    0.0 0.07952763828001923 \\
1.168001611 0.003640662762086863 \\
2.147860381 0.0009963442919841965 \\
3.165431995 5.94749658471245e-5 \\
4.207980257 7.896543742392376e-6 \\
5.161953094 2.09488566270105e-6 \\
6.122767699000001 8.088654121203195e-8 \\
7.089030559 4.894990587913975e-9 \\
};
\addlegendentry {EARGD$ \qquad$}

    \addplot[color=cpl2, draw opacity={1.0}, line width={1}, solid, mark=|, mark repeat=2, mark phase=2]
table[row sep={\\}]
      {
    \\
    0.0 0.07952763828001923 \\
0.48118933500000005 0.08692498693968086 \\
0.859881911 0.0717820452237402 \\
1.2296436590000002 0.05923506635383784 \\
1.600387519 0.04568088364542617 \\
1.9703317880000002 0.037876706561778724 \\
2.339164588 0.03403750128353163 \\
2.71000406 0.03175569126264674 \\
3.073628229 0.028856473192861446 \\
3.441436626 0.02478629669947618 \\
3.8344872820000004 0.02139382587341938 \\
4.201792354 0.018547730622794464 \\
4.568625143 0.016355641841361162 \\
5.147950428000001 0.01457968032720044 \\
5.511721586 0.013156095936498073 \\
5.877029933 0.011806136943779659 \\
6.243013521 0.010516765732818657 \\
6.614787062 0.009068800529681504 \\
6.980165468 0.007678115610052347 \\
7.35070227 0.0065720757049602995 \\
7.722946329000001 0.005761981138397605 \\
8.096370697000001 0.005198642692827996 \\
8.468337225 0.004797646194898956 \\
8.841446326 0.004373982843415201 \\
9.211931332 0.003917324558918743 \\
9.583809529 0.003469697284755493 \\
9.960764593 0.0030329596817565163 \\
10.327851936 0.0026905604886783836 \\
10.704807146 0.0024067473081049538 \\
11.076302153 0.0021667377031614403 \\
11.447889 0.001977807200160992 \\
11.817985781 0.001788141079149819 \\
12.188215591 0.0016386920712307478 \\
12.562847402000001 0.0015098267597637605 \\
12.937708708 0.0013631774036882858 \\
13.313563674000001 0.0012274405099960074 \\
13.682819843 0.001107728308850681 \\
14.054714982 0.001009907187053467 \\
14.426335470000001 0.0009414032232838619 \\
14.798546149000002 0.0008896781857227639 \\
15.169537179 0.000848003083252443 \\
15.536511484000002 0.0008189791263007692 \\
15.910000983000002 0.0007801739105197104 \\
16.278965024 0.0007347632979382547 \\
16.649225065 0.0006907706912073304 \\
17.019297327 0.0006553451865587153 \\
17.392852299 0.0006172023317776181 \\
17.767249485 0.000583119917488437 \\
18.135871853 0.0005506190722714602 \\
18.50760073 0.0005087748608089967 \\
18.877823857000003 0.00046550057485192494 \\
19.246764759 0.00041625173909033057 \\
19.616843185 0.00037612461637198685 \\
19.987062096000002 0.0003404890901074389 \\
20.356577823000002 0.0003116239494612655 \\
20.729255259000002 0.00028459625127900696 \\
21.100137699 0.00026050484711647446 \\
21.469510686 0.00023925203777669894 \\
21.839467273 0.00021504972436908631 \\
22.210596992000003 0.0001929232750592563 \\
22.579983864000003 0.00017191214036453822 \\
22.950595647 0.0001539199905697635 \\
23.318502292 0.0001394210052777111 \\
23.688220774 0.00012447804781728054 \\
24.057834864 0.00011212023356855171 \\
24.427582035 0.00010111624205482681 \\
24.802361437000002 9.159436668269509e-5 \\
25.174100698 8.315159911794635e-5 \\
25.542954075 7.568068513103377e-5 \\
25.910465369 6.91803791065654e-5 \\
26.278478080000003 6.274615909435801e-5 \\
26.648898473000003 5.672641264041971e-5 \\
27.019911433 5.213093726056359e-5 \\
27.387872269000002 4.900397457591007e-5 \\
27.759728318 4.631368331908214e-5 \\
28.130193666 4.388550590231333e-5 \\
28.503808079000002 4.059820221843185e-5 \\
28.873923311000002 3.7513847584416435e-5 \\
29.247473588000002 3.52230614189692e-5 \\
29.618305135000004 3.2728384197680434e-5 \\
29.991413498000004 3.110628390624551e-5 \\
30.362345577000003 2.893320320737647e-5 \\
30.733760865 2.7032008576532294e-5 \\
31.110930955 2.5142326335536702e-5 \\
31.481636268000003 2.3207619373365465e-5 \\
31.85325508 2.1536498090057516e-5 \\
32.223864446 2.0151022075160345e-5 \\
32.592545543 1.8649318044088496e-5 \\
32.961028752000004 1.744165380442338e-5 \\
33.327487462 1.6216129087192333e-5 \\
33.698802026 1.4957607830768397e-5 \\
34.074831652 1.3529304506863862e-5 \\
34.444629648 1.2377570838467982e-5 \\
34.816766771000005 1.1260865310442447e-5 \\
35.186353044 1.0114031309798533e-5 \\
35.556850789 9.152729152529875e-6 \\
35.925668975 8.088924868743557e-6 \\
36.295950834 7.476138922229348e-6 \\
36.666238564000004 6.948221192284723e-6 \\
37.040742890000004 6.32207151309887e-6 \\
37.408343924 5.810182666295277e-6 \\
};
\addlegendentry {L2RCG$ \qquad$}

    \addplot[color=cpl3, draw opacity={1.0}, line width={1}, solid, mark=diamond, mark repeat=2, mark phase=2]
table[row sep={\\}]
{
    \\
    0.0 0.07952763828001923 \\
1.444919987 0.12893725317567875 \\
2.1218558030000003 0.08054244118795088 \\
2.946389674 0.026600558687113734 \\
3.959060474 0.0030954432418888733 \\
4.718925863 0.0011033984790520708 \\
5.502516249 0.000425203289829024 \\
6.640613004 1.4815006321120977e-5 \\
7.4326631590000005 3.7793131707142958e-6 \\
8.200089995 1.4887599971880954e-6 \\
8.917947538 5.474512799013683e-7 \\
9.983463593 3.007635431783204e-8 \\
10.812061446000001 8.068317825399006e-9 \\
11.488485679 4.304868154996554e-9 \\
};
\addlegendentry {SCF$ \qquad$}

    \addplot[color=cpl4, draw opacity={1.0}, line width={1}, solid, mark=asterisk, mark repeat=2, mark phase=2]
table[row sep={\\}]
{
    \\0.0 0.07952763828001923 \\
1.27855662 0.058329767116307926 \\
2.017738617 0.039798892409636244 \\
2.947656405 0.017805048515566677 \\
3.6926095890000004 0.011528001940082815 \\
4.4325698870000005 0.006948851130300159 \\
5.4342625 0.004233476436237086 \\
6.176282348 0.0023182806030574835 \\
6.92116001 0.0014205282915731611 \\
7.662349559000001 0.0008150285950397912 \\
8.400206427 0.0004548815823840144 \\
9.143888273 0.00027603513489712005 \\
9.882410440000001 0.00016030650103703606 \\
10.6188374 0.00010244497311644544 \\
11.361204616 5.925955266030251e-5 \\
12.099862086 3.4340292534034424e-5 \\
13.092234534000001 2.0877631354642103e-5 \\
13.857567581000001 1.2385292559899828e-5 \\
14.599399586 7.651625450363235e-6 \\
15.343597844000001 4.802576794391758e-6 \\
16.336439575 3.007914473851147e-6 \\
17.329352447 1.8628917432171206e-6 \\
18.337092235 1.171449846388138e-6 \\
19.077394807 7.354057375278773e-7 \\
20.075177273 4.699879769780677e-7 \\
20.818117869 2.661941890748651e-7 \\
21.556509803 1.5985203908809032e-7 \\
22.296458657000002 9.234193837237451e-8 \\
23.03643238 5.07289986634458e-8 \\
23.773756198 2.869390629966284e-8 \\
24.509850996 1.5604452373072072e-8 \\
25.249938406000002 8.888796301223622e-9 \\
25.996992528000003 4.866734199222215e-9 \\
};
\addlegendentry {LBFGS}
\end{axis}
   \end{tikzpicture}
    \caption{Silicon model. EARGD and EARCG outperform the other three methods  with respect to all considered criteria. L2RCG exhibits the poorest performance.
    }
    \label{fig:plots-silicon}\vspace*{5mm}
\end{subfigure}
\begin{subfigure}{\textwidth}
\centering
    \begin{tikzpicture}[/tikz/background rectangle/.style={fill={rgb,1:red,1.0;green,1.0;blue,1.0}, fill opacity={1.0}, draw opacity={1.0}}, show background rectangle]

\begin{axis}[
name = gaas1,
width={\plotwidth}, height={\plotheight}, 
at={(0.0mm,0.0mm)},
xlabel={Iterations}, 
xmin={0}, xmax={30},
xtick align={inside}, 
x tick style={color={rgb,1:red,0.0;green,0.0;blue,0.0}, opacity={1.0}}, 
xmajorgrids={true},
x grid style={color={rgb,1:red,0.0;green,0.0;blue,0.0}, draw opacity={0.1}, line width={0.5}, solid}, 
scaled y ticks={false}, 
ylabel={$\|R^{(m)}\|_F$}, 
ylabel style={font={{\fontsize{11 pt}{14.3 pt}\selectfont}}, color={rgb,1:red,0.0;green,0.0;blue,0.0}, draw opacity={1.0}, rotate={0.0}}, 
ymode={log}, log basis y={10}, ymin={5e-9}, ymax={0.2}, 
yticklabels=\empty, 
ytick={{1.0e-8,1.0e-6,0.0001,0.01}}, 
ytick align={inside}, 
y tick style={color={rgb,1:red,0.0;green,0.0;blue,0.0}, opacity={1.0}}, 
ymajorgrids={true}, 
y grid style={color={rgb,1:red,0.0;green,0.0;blue,0.0}, draw opacity={0.1}, line width={0.5}, solid}, 
colorbar={false}
]

\addplot[color=unia-purple, draw opacity={1.0}, line width={1.5}, solid, mark=square, mark repeat=2,mark phase=2]
table[row sep={\\}]
{
    \\
    0 0.07000296681563047 \\
1 0.025222327832397175 \\
2 0.001066748436174278 \\
3 0.0003379265519915209 \\
4 7.614973087021073e-5 \\
5 3.0162737758155897e-5 \\
6 7.194624412576693e-6 \\
7 1.110659943256904e-6 \\
8 1.1953473416072252e-7 \\
9 2.212524755767015e-8 \\
10 2.8408499624486333e-9 \\
};

\addplot[color=cpl1, draw opacity={1.0}, line width={1}, solid, mark=o, mark repeat=2,mark phase=2]
table[row sep={\\}]
{
    \\
    0 0.07000296681563047 \\
1 0.04513600631521441 \\
2 0.003120130347895161 \\
3 0.0004790301555918699 \\
4 0.00024161924537239034 \\
5 2.682119408275478e-5 \\
6 6.27879255157347e-6 \\
7 5.3727777437361545e-6 \\
8 5.198228375870221e-6 \\
9 1.427335547748887e-6 \\
10 4.2470299297959264e-7 \\
11 3.802338665981345e-8 \\
12 1.983546489346481e-8 \\
13 1.2073003724602435e-8 \\
14 8.34204272018569e-9 \\
};

    \addplot[color=cpl2, draw opacity={1.0}, line width={1}, solid, mark=|, mark repeat=2, mark phase=2]
table[row sep={\\}]
      {
    \\
    0 0.07000296681563047 \\
1 0.10338939827931397 \\
2 0.09568018323737128 \\
3 0.07173069570111118 \\
4 0.05526079985941775 \\
5 0.04774125929341334 \\
6 0.0385584201364543 \\
7 0.030658266616156303 \\
8 0.028287891299678247 \\
9 0.02551624163687259 \\
10 0.024645068215622977 \\
11 0.023435210632089885 \\
12 0.022769500656120715 \\
13 0.021708464521560177 \\
14 0.01944544990095817 \\
15 0.01726949416770217 \\
16 0.015276410592385917 \\
17 0.013768652857253549 \\
18 0.013100500287487582 \\
19 0.012879754345531918 \\
20 0.012563690752547577 \\
21 0.011956341595869262 \\
22 0.011240636009028518 \\
23 0.010664340435771427 \\
24 0.010322150302167989 \\
25 0.009625775102526487 \\
26 0.009051190385046546 \\
27 0.008226628656806304 \\
28 0.007381320202832916 \\
29 0.006508390374756113 \\
30 0.005759112729582509 \\
31 0.0051915999823872955 \\
32 0.004558447663149401 \\
33 0.0040357583514895375 \\
34 0.0036901070200236457 \\
35 0.0033084553160707434 \\
36 0.0030696973316451455 \\
37 0.0028934088318896783 \\
38 0.002690826202277596 \\
39 0.0025062213407462805 \\
40 0.0023903226288457405 \\
41 0.0022580528604272234 \\
42 0.001994003305058301 \\
43 0.0017513448503942417 \\
44 0.001573078768371541 \\
45 0.0014372593227758507 \\
46 0.0013379894013392166 \\
47 0.0012299554090857855 \\
48 0.0011377531743433326 \\
49 0.0010438134411413231 \\
50 0.0009536095380278964 \\
51 0.0008644025710301415 \\
52 0.0007849667581718376 \\
53 0.000729240833790316 \\
54 0.0006569027744637063 \\
55 0.000620512338875396 \\
56 0.0005800561605557 \\
57 0.0005515549150921965 \\
58 0.0005212421838933913 \\
59 0.0004973641813835729 \\
60 0.000502958980748215 \\
61 0.0004888247902383571 \\
62 0.00047963993636104374 \\
63 0.00048220556079589543 \\
64 0.0004650918251180437 \\
65 0.00045409847003410555 \\
66 0.00047200355036010496 \\
67 0.0004887775048912001 \\
68 0.000523341768908648 \\
69 0.0005347899472286706 \\
70 0.0005587371431846693 \\
71 0.0005854513553119494 \\
72 0.000607992921416189 \\
73 0.0006098073496057505 \\
74 0.000624449903758525 \\
75 0.0006354376033621258 \\
76 0.0006580912500037923 \\
77 0.0007149233576171275 \\
78 0.0007515348551311185 \\
79 0.0007682078343547905 \\
80 0.0008054951984012285 \\
81 0.0008274343853036825 \\
82 0.0008196077387927248 \\
83 0.0007842870273372841 \\
84 0.0007135311600999446 \\
85 0.0006341080338987864 \\
86 0.0005907887010358888 \\
87 0.0005629322899463482 \\
88 0.0005454121615644804 \\
89 0.0005144460006779876 \\
90 0.0004769794649274936 \\
91 0.0004380241827068045 \\
92 0.00041245998202327666 \\
93 0.0003854918146005144 \\
94 0.0003620462998249835 \\
95 0.00033190408070762754 \\
96 0.0003048978330550381 \\
97 0.00027873719719727623 \\
98 0.0002557178109524686 \\
99 0.00023869573483783482 \\
100 0.00023390008995705878 \\
};

    \addplot[color=cpl3, draw opacity={1.0}, line width={1}, solid, mark=diamond, mark repeat=2, mark phase=2]
table[row sep={\\}]
{
    \\
    0 0.07000296681563047 \\
1 0.04176287642392045 \\
2 0.05298818744782466 \\
3 0.00939403913998692 \\
4 0.008221196264033979 \\
5 0.00039080152680717273 \\
6 7.95277562181964e-5 \\
7 8.354925487958839e-6 \\
8 1.0226485599333353e-6 \\
9 8.099338846952489e-8 \\
10 3.175480430858788e-8 \\
11 1.9125879682153224e-9 \\
12 7.197061122621676e-10 \\
};

    \addplot[color=cpl4, draw opacity={1.0}, line width={1}, solid, mark=asterisk, mark repeat=2, mark phase=2]
table[row sep={\\}]
{
    \\0 0.07000296681563047 \\
1 0.029759963614647906 \\
2 0.026942415175809607 \\
3 0.012706078949551811 \\
4 0.006552589310215914 \\
5 0.0033242105127793647 \\
6 0.0022509461329942404 \\
7 0.001201463984899357 \\
8 0.0006290088009380205 \\
9 0.0004882928932446588 \\
10 0.0005424312627564232 \\
11 0.0007308203862130496 \\
12 0.0008088683249002766 \\
13 0.0005083013673964215 \\
14 0.00025000773436973745 \\
15 0.00023132900089011934 \\
16 0.00019262111910996793 \\
17 0.0001297761055415724 \\
18 8.428882926342223e-5 \\
19 5.6406963254074596e-5 \\
20 3.109220954531755e-5 \\
21 1.7590814851474398e-5 \\
22 7.694522317772446e-6 \\
23 3.6350551263051523e-6 \\
24 1.9370022535695948e-6 \\
25 1.1305461507220519e-6 \\
26 7.645438219982155e-7 \\
27 4.487475735904333e-7 \\
28 2.1563137715693183e-7 \\
29 1.5953395981834576e-7 \\
30 1.5147471525857652e-7 \\
31 1.8969176155485392e-7 \\
32 1.9495361561110883e-7 \\
33 1.9073890909670553e-7 \\
34 1.3507306086596036e-7 \\
35 1.0101305613219568e-7 \\
36 8.571911863942167e-8 \\
37 6.689554606536155e-8 \\
38 4.8278642577977574e-8 \\
};
\end{axis}


\begin{axis}[
name = gaas2,
at={(51.5mm,0.0mm)}, 
width={\plotwidth}, height={\plotheight},
xlabel={Hamiltonians}, 
xmin={0}, xmax={320}, 
xtick align={inside}, 
x tick style={color={rgb,1:red,0.0;green,0.0;blue,0.0}, opacity={1.0}}, 
xmajorgrids={true}, 
x grid style={color={rgb,1:red,0.0;green,0.0;blue,0.0}, draw opacity={0.1}, line width={0.5}, solid}, 
ymode={log}, log basis y={10},
ymin={5e-9}, ymax={0.2}, 
scaled y ticks={false}, 
yticklabels={{$10^{-8}$,$10^{-6}$,$10^{-4}$,$10^{-2}$}}, 
ytick={{1.0e-8,1.0e-6,0.0001,0.01}}, 
ytick align={inside}, 
y tick style={color={rgb,1:red,0.0;green,0.0;blue,0.0}, opacity={1.0}}, 
ymajorgrids={true}, 
y grid style={color={rgb,1:red,0.0;green,0.0;blue,0.0}, draw opacity={0.1}, line width={0.5}, solid}, 
colorbar={false}
]

\addplot[color=unia-purple, draw opacity={1.0}, line width={1.5}, solid, mark=square, mark repeat=2,mark phase=2]
table[row sep={\\}]
{
    \\
    0.0 0.07000296681563047 \\
36.0 0.025222327832397175 \\
60.0 0.001066748436174278 \\
84.0 0.0003379265519915209 \\
112.0 7.614973087021073e-5 \\
140.0 3.0162737758155897e-5 \\
165.0 7.194624412576693e-6 \\
191.0 1.110659943256904e-6 \\
216.0 1.1953473416072252e-7 \\
240.0 2.212524755767015e-8 \\
271.0 2.8408499624486333e-9 \\
};

\addplot[color=cpl1, draw opacity={1.0}, line width={1}, solid, mark=o, mark repeat=2,mark phase=2]
table[row sep={\\}]
{
    \\
    0.0 0.07000296681563047 \\
33.0 0.04513600631521441 \\
61.0 0.003120130347895161 \\
83.0 0.0004790301555918699 \\
106.0 0.00024161924537239034 \\
131.0 2.682119408275478e-5 \\
154.0 6.27879255157347e-6 \\
177.0 5.3727777437361545e-6 \\
201.0 5.198228375870221e-6 \\
221.0 1.427335547748887e-6 \\
241.0 4.2470299297959264e-7 \\
263.0 3.802338665981345e-8 \\
286.0 1.983546489346481e-8 \\
309.0 1.2073003724602435e-8 \\
331.0 8.34204272018569e-9 \\
};

    \addplot[color=cpl2, draw opacity={1.0}, line width={1}, solid, mark=|, mark repeat=2, mark phase=2]
table[row sep={\\}]
      {
    \\
    0.0 0.07000296681563047 \\
9.0 0.10338939827931397 \\
15.0 0.09568018323737128 \\
21.0 0.07173069570111118 \\
27.0 0.05526079985941775 \\
33.0 0.04774125929341334 \\
39.0 0.0385584201364543 \\
45.0 0.030658266616156303 \\
51.0 0.028287891299678247 \\
57.0 0.02551624163687259 \\
63.0 0.024645068215622977 \\
69.0 0.023435210632089885 \\
75.0 0.022769500656120715 \\
81.0 0.021708464521560177 \\
87.0 0.01944544990095817 \\
93.0 0.01726949416770217 \\
99.0 0.015276410592385917 \\
105.0 0.013768652857253549 \\
111.0 0.013100500287487582 \\
117.0 0.012879754345531918 \\
123.0 0.012563690752547577 \\
129.0 0.011956341595869262 \\
135.0 0.011240636009028518 \\
141.0 0.010664340435771427 \\
147.0 0.010322150302167989 \\
153.0 0.009625775102526487 \\
159.0 0.009051190385046546 \\
165.0 0.008226628656806304 \\
171.0 0.007381320202832916 \\
177.0 0.006508390374756113 \\
183.0 0.005759112729582509 \\
189.0 0.0051915999823872955 \\
195.0 0.004558447663149401 \\
201.0 0.0040357583514895375 \\
207.0 0.0036901070200236457 \\
213.0 0.0033084553160707434 \\
219.0 0.0030696973316451455 \\
225.0 0.0028934088318896783 \\
231.0 0.002690826202277596 \\
237.0 0.0025062213407462805 \\
243.0 0.0023903226288457405 \\
249.0 0.0022580528604272234 \\
255.0 0.001994003305058301 \\
261.0 0.0017513448503942417 \\
267.0 0.001573078768371541 \\
273.0 0.0014372593227758507 \\
279.0 0.0013379894013392166 \\
285.0 0.0012299554090857855 \\
291.0 0.0011377531743433326 \\
297.0 0.0010438134411413231 \\
303.0 0.0009536095380278964 \\
309.0 0.0008644025710301415 \\
315.0 0.0007849667581718376 \\
321.0 0.000729240833790316 \\
327.0 0.0006569027744637063 \\
333.0 0.000620512338875396 \\
339.0 0.0005800561605557 \\
345.0 0.0005515549150921965 \\
351.0 0.0005212421838933913 \\
357.0 0.0004973641813835729 \\
363.0 0.000502958980748215 \\
369.0 0.0004888247902383571 \\
375.0 0.00047963993636104374 \\
381.0 0.00048220556079589543 \\
387.0 0.0004650918251180437 \\
393.0 0.00045409847003410555 \\
399.0 0.00047200355036010496 \\
405.0 0.0004887775048912001 \\
411.0 0.000523341768908648 \\
417.0 0.0005347899472286706 \\
423.0 0.0005587371431846693 \\
429.0 0.0005854513553119494 \\
435.0 0.000607992921416189 \\
441.0 0.0006098073496057505 \\
447.0 0.000624449903758525 \\
453.0 0.0006354376033621258 \\
459.0 0.0006580912500037923 \\
465.0 0.0007149233576171275 \\
471.0 0.0007515348551311185 \\
477.0 0.0007682078343547905 \\
483.0 0.0008054951984012285 \\
489.0 0.0008274343853036825 \\
495.0 0.0008196077387927248 \\
501.0 0.0007842870273372841 \\
507.0 0.0007135311600999446 \\
513.0 0.0006341080338987864 \\
519.0 0.0005907887010358888 \\
525.0 0.0005629322899463482 \\
531.0 0.0005454121615644804 \\
537.0 0.0005144460006779876 \\
543.0 0.0004769794649274936 \\
549.0 0.0004380241827068045 \\
555.0 0.00041245998202327666 \\
561.0 0.0003854918146005144 \\
567.0 0.0003620462998249835 \\
573.0 0.00033190408070762754 \\
579.0 0.0003048978330550381 \\
585.0 0.00027873719719727623 \\
591.0 0.0002557178109524686 \\
597.0 0.00023869573483783482 \\
603.0 0.00023390008995705878 \\
};

    \addplot[color=cpl3, draw opacity={1.0}, line width={1}, solid, mark=diamond, mark repeat=2, mark phase=2]
table[row sep={\\}]
{
    \\
    0.0 0.07000296681563047 \\
38.0 0.04176287642392045 \\
48.5 0.05298818744782466 \\
65.0 0.00939403913998692 \\
80.5 0.008221196264033979 \\
99.25 0.00039080152680717273 \\
115.75 7.95277562181964e-5 \\
133.75 8.354925487958839e-6 \\
150.0 1.0226485599333353e-6 \\
167.75 8.099338846952489e-8 \\
184.75 3.175480430858788e-8 \\
201.75 1.9125879682153224e-9 \\
219.25 7.197061122621676e-10 \\
};

    \addplot[color=cpl4, draw opacity={1.0}, line width={1}, solid, mark=asterisk, mark repeat=2, mark phase=2]
table[row sep={\\}]
{
    \\0.0 0.07000296681563047 \\
12.0 0.029759963614647906 \\
18.0 0.026942415175809607 \\
24.0 0.012706078949551811 \\
30.0 0.006552589310215914 \\
36.0 0.0033242105127793647 \\
42.0 0.0022509461329942404 \\
48.0 0.001201463984899357 \\
54.0 0.0006290088009380205 \\
63.0 0.0004882928932446588 \\
72.0 0.0005424312627564232 \\
81.0 0.0007308203862130496 \\
90.0 0.0008088683249002766 \\
99.0 0.0005083013673964215 \\
108.0 0.00025000773436973745 \\
117.0 0.00023132900089011934 \\
126.0 0.00019262111910996793 \\
135.0 0.0001297761055415724 \\
141.0 8.428882926342223e-5 \\
147.0 5.6406963254074596e-5 \\
153.0 3.109220954531755e-5 \\
159.0 1.7590814851474398e-5 \\
165.0 7.694522317772446e-6 \\
171.0 3.6350551263051523e-6 \\
177.0 1.9370022535695948e-6 \\
183.0 1.1305461507220519e-6 \\
189.0 7.645438219982155e-7 \\
195.0 4.487475735904333e-7 \\
201.0 2.1563137715693183e-7 \\
210.0 1.5953395981834576e-7 \\
219.0 1.5147471525857652e-7 \\
228.0 1.8969176155485392e-7 \\
237.0 1.9495361561110883e-7 \\
246.0 1.9073890909670553e-7 \\
255.0 1.3507306086596036e-7 \\
264.0 1.0101305613219568e-7 \\
273.0 8.571911863942167e-8 \\
282.0 6.689554606536155e-8 \\
288.0 4.8278642577977574e-8 \\
};
\end{axis}

         \input{figs/GaAs-plt3}
   \end{tikzpicture}
    \caption{GaAs model. EARCG, EARGD are slightly worse than SCF and are superior to that of LBFGS and L2RCG.
    }
    \label{fig:plots-GaAs}\vspace*{5mm}
\end{subfigure}
\begin{subfigure}{\textwidth}
\centering
    \begin{tikzpicture}[/tikz/background rectangle/.style={fill={rgb,1:red,1.0;green,1.0;blue,1.0}, fill opacity={1.0}, draw opacity={1.0}}, show background rectangle]
         \input{figs/TiO2-plt1}
        \input{figs/TiO2-plt2}
         \input{figs/TiO2-plt3}
   \end{tikzpicture}
    \caption{TiO$_2$ model. EARCG outperforms the other methods with respect to the iteration number and CPU time, while LBFGS shows a similar  performance as EARCG in terms of Hamiltonian applications. Here, the superiority of EARCG over EARGD is the largest among the models under consideration.}
    \label{fig:plots-TiO2}
\end{subfigure}
\caption{Performance comparison of the different methods in terms of iterations, Hamiltonian operator applications and CPU time for three different models.}
\end{figure}

To evaluate the performance of the methods, we consider two different measures, the CPU time and the number of applications of the Hamiltonian operator relative to the Frobenius norm of the residual. The number of Hamiltonian applications is easy to count and not affected by hardware, implementation details, and background processes.
Table~\ref{tab:ham_call_percentage} shows the percentage of total runtime caused by Hamiltonian applications in our experiments. To make the comparison more accurate, we do not count the calls of \texttt{DftHamiltonian\_multiplication} directly, but normalize them with respect to the number of bands, since SCF often uses more bands than the other methods. One can see that the percentages for EARCG, SCF, and L2RCG are mostly of the same order of magnitude, while those for LBFGS are slightly lower.
In addition to these usual performance measures, we also include the number of (outer) iterations of the methods for illustrative purposes. 

\begin{table}[H]
\begin{center}
\begin{tabular}{||c | c c c c||} 
 \hline
   & EARCG & L2RCG & SCF & LBFGS \\ [0.5ex] 
 \hline\hline
 Silicon & 59.7\% &56.3\% & 56.5\% & 29.0 \% \\
 GaAs & 49.4 \% & 37.4\% & 48.8\% & 17.5\% \\
 TiO$_2$  & 81.5\% &  70.6\% & 68.2\% & 38.0\% \\ [1ex] 
 \hline
\end{tabular}
\end{center}
    \caption{Percentage of overall runtime cost caused by \texttt{DftHamiltonian\_multiplication}.}
    \label{tab:ham_call_percentage}
\end{table}

The Figures~\ref{fig:plots-silicon}, \ref{fig:plots-GaAs} and \ref{fig:plots-TiO2} show the results of the performance test of the methods for the different model solids. The $y$ axis always represents the Frobenius norm of the residual $R^{(m)}=H_{\varPhi^{(m)}}\varPhi^{(m)}-\varPhi^{(m)}\big((\varPhi^{(m)})^*H_{\varPhi^{(m)}}\varPhi^{(m)}\big)$ at iteration $m$, while the $x$-axes represent $m$ (left plots), the number of Hamiltonian applications after~$m$ steps (middle plots), and the CPU time spent after~$m$ steps (right plots). 

In terms of Hamiltonian applications, we observe that EARCG performs at least as well as SCF, EARGD, and LBFGS, and is much more efficient than L2RCG without preconditioning. In terms of CPU time, EARCG, EARGD, and SCF all perform similarly, while LBFGS is significantly worse due to the overhead caused by the number of steps and the expensive step size calculation.  It is worth noting that our DFTK implementations of the energy-adaptive EARCG and EARGD methods are faster than the DFTK built-in optimized SCF variant for all test cases in the low-accuracy regime. 

For the silicon model, the two energy-adaptive methods behave similarly and significantly outperform the other methods for all measures in all accuracy regimes. For the GaAs model, EARCG has a minimal edge over EARGD, and both perform slightly worse than SCF at high accuracy $10^{-8}$, outperforming LBFGS and L2RCG significantly. The performance gain of EARCG over EARGD is most pronounced for the TiO$_2$ model, where EARCG outperforms EARGD by more than 10\% for all measures.

\section{Conclusion}\label{sec:concl}
In this paper, we extended the energy-adaptive Riemannian gradient descent (EARGD) approach of \cite{AltPS22} to a conjugate gradient scheme, thereby developing the energy-adaptive Riemannian conjugate gradient (EARCG) method. We also discussed strategies for shifting the Hamiltonian and solving the resulting shifted system. In experiments on small lattice solid-state systems, our method shows superior performance compared to the $L^2$-gradient RCG method without preconditioning, is competitive with state-of-the-art SCF-based methods, and demonstrates particularly promising gains in small-gap systems compared to other Riemannian methods. Furthermore, EARCG is an improvement over its gradient descent counterpart, with the performance difference becoming more pronounced for larger models. 

We attribute the improved performance to the method's dual exploitation of the manifold structure of $\St$ and the specific structure of Kohn-Sham eigenvalue problems through the use of a gap-aware Riemannian metric based on the shifted Hamiltonian. In particular, the improvements introduced here also benefit the energy-adaptive Riemannian gradient descent approach.

Looking ahead, several directions seem promising. These include formulating RCG methods for alternative metrics, such as the simpler $H^1$-metric, and extending energy-adaptive Riemannian methods to LBFGS or other advanced gradient-based algorithms, which could further improve the performance of Riemannian optimization for Kohn-Sham problems. In addition, improving the efficiency of energy-adaptive Riemannian gradient computations and integrating widely used SCF heuristics, such as potential mixing and adaptive band strategies, could further reduce the computational cost for larger and more complex systems. The energy-adaptive framework also has great potential in interaction with adaptive spatial discretization \cite{HSW21}, which may be utilizable by EARCG. Finally, the incorporation of fractional orbital occupancy would allow the application of EARCG to metallic systems and non-zero temperature scenarios, although its integration into a Riemannian framework remains challenging.

\bigskip\noindent
\textbf{Acknowledgments} The authors thank Michael Herbst and Gaspard Kemlin for fruitful and constructive discussions on mathematical aspects, as well as algorithmic and implementational details.

\bibliography{bibliography.bib}


\appendix

\section{Preconditioned block FOM}\label{app:FOM}
For the sake of completeness and reproducibility, we present the preconditioned block FOM to solve the Sylvester equation~\eqref{eq:Sylv}, which is a~block Krylov subspace method \cite{Simoncini1996}.
The algorithm, a preconditioned version of \cite[Algorithm~3.3]{Robbe2002}, is given in Algorithm~\ref{alg:pfom}.

\begin{algorithm2e}[th]
\SetKwInOut{Input}{input}
\SetAlgoLined
\Input{$\varPhi \in \Field^{n \times \pp}$, $H_\varPhi \in \Herm(n)$, $\Sigma_{\varPhi,\mu} \in \Herm(\pp)$, 
initial guess $X_0 \in \Field^{n \times \pp}$,
preconditioner $P^{-1}: \Field^{n\times \pp} \to \Field^{n \times \pp}$}
\BlankLine

$C = \varPhi - (H_\varPhi X_0 - X_0\Sigma_{\varPhi,\mu})$\; 
$R_0 = P^{-1} C$\;
$ V_1 \in \Field^{n \times q_1}$, $q_1=\tilde q_1 \le p$, is the unitary factor 
of the rank-revealing QR decomposition of $R_0$\;
$W_1 = H_\varPhi V_1$\;
$A_1 =  V_1^\ast  W_1^{}$\;
$C_1 = V_1^\ast  C$\;
 \For{$j = 1, 2, \dots$ until convergence}{
   solve $A_jY_j - Y_j \Sigma_{\varPhi,\mu} = C_j$ for $Y_j \in \Field^{q_j \times \pp}$\;
  $R_j = (I - V_j V_j^\ast )P^{-1}(C - (W_jY_j - V_j Y_j\Sigma_{\varPhi,\mu}))$\;
    $\tilde V_{j+1} \in \Field^{n \times \tilde q_{j+1}}$, $\tilde q_{j+1} \le p$, is the unitary factor of the rank-revealing QR decomposition of $R_j$\;

  $V_{j+1} = \begin{bmatrix}
      V_{j} & \tilde V_{j+1}
  \end{bmatrix}$\;
  $W_{j+1} = \begin{bmatrix}
      W_{j} & H_\varPhi \tilde V_{j+1}^{}
    \end{bmatrix}$\;
      $H_{j, j+1} = V_j^\ast H_\varPhi\tilde V_{j+1}$\;
  $H_{j+1,j+1} = \tilde V_{j+1}^\ast H_\varPhi\tilde V_{j+1}$\;
  $A_{j+1} = \begin{bmatrix}
      A_{j} & H_{j, j+1} \\
 H_{j,j+1}^\ast 
 & H_{j+1,j+1}
  \end{bmatrix}\in\Herm(q_{j+1})$ with $q_{j+1} = \sum\limits_{i=1}^{j+1} \tilde q_i$ rows\;
  $C_{j+1} = \begin{bmatrix}
      C_j \\ \tilde V_j^\ast  C 
  \end{bmatrix}\in\Field^{q_{j+1}\times \pp}$ 
 }
 \KwRet{$X_j=X_0 + V_j Y_j$}
 \caption{Preconditioned block FOM for the Sylvester equation}
     \label{alg:pfom}
\end{algorithm2e}

\end{document}